\documentclass[12pt, reqno]{amsart}
\usepackage{amsmath, amsthm, amscd, amsfonts, amssymb, graphicx, xcolor}
\usepackage[bookmarksnumbered, colorlinks, plainpages]{hyperref}
\usepackage{dsfont}
\usepackage{tikz}
\usetikzlibrary{shapes.misc} 
\tikzset{cross/.style={cross out, draw=red, minimum size=2*(#1-\pgflinewidth), inner sep=0pt, outer sep=0pt},
cross/.default={1pt}}
\usetikzlibrary{automata,positioning}
\usepackage{subcaption}
\usepackage{mathtools}
\mathtoolsset{showonlyrefs=true}
\usepackage{multirow,multicol}
\usepackage{grffile} 
\usepackage{ulem, cancel}
\allowdisplaybreaks

\usepackage[left=2.5cm,right=2.5cm,top=2.5cm,bottom=2.5cm]{geometry}

\newtheorem{theorem}{Theorem}[section]
\newtheorem{lemma}[theorem]{Lemma}
\newtheorem{proposition}[theorem]{Proposition}
\newtheorem{corollary}[theorem]{Corollary}
\theoremstyle{definition}
\theoremstyle{remark}

\numberwithin{equation}{section}
\newtheorem*{assumption}{Assumption}
\renewcommand{\assumption}[2]{\textit{Assumption}\hspace{1pt}#1: #2}

\RequirePackage[backend=biber,style=apa]{biblatex}
\newcommand{\citelink}[2]{\hyperlink{cite.\therefsection @#1}{\textcolor{blue}{#2}}}
\newcommand{\dd}{\mathop{}\mathopen{}\mathrm{d}}

\newcommand{\cercle}{\mathbb{S}^1}
\newcommand{\Inte}{\int_{\cercle}}
\newcommand{\argmin}[1]{\underset{#1}{\mathrm{argmin}}}
\newcommand{\indices}{\lambda\in \Lambda_m}
\newcommand{\philambda}{\varphi_{\lambda}}
\newcommand{\Proba}{\mathbb P}
\newcommand{\Esp}{\mathbb E}
\DeclareMathOperator{\Var}{Var}

\makeatletter

\renewcommand{\tocsection}[3]{%
  \indentlabel{\@ifnotempty{#2}{\bfseries\ignorespaces#1 #2\quad}}\bfseries#3}
\renewcommand{\tocsubsection}[3]{%
  \indentlabel{%
    \hbox to 3em{#2\hfil}%
  }#3%
}

\newcommand\@dotsep{4.5}
\def\@tocline#1#2#3#4#5#6#7{\relax
  \ifnum #1>\c@tocdepth 
  \else
    \par \addpenalty\@secpenalty\addvspace{#2}%
    \begingroup \hyphenpenalty\@M
    \@ifempty{#4}{%
      \@tempdima\csname r@tocindent\number#1\endcsname\relax
    }{%
      \@tempdima#4\relax
    }%
    \parindent\z@ \leftskip#3\relax \advance\leftskip\@tempdima\relax
    \rightskip\@pnumwidth plus1em \parfillskip-\@pnumwidth
    #5\leavevmode\hskip-\@tempdima{#6}\nobreak
    \leaders\hbox{$\m@th\mkern \@dotsep mu\hbox{.}\mkern \@dotsep mu$}\hfill
    \nobreak
    \hbox to\@pnumwidth{\@tocpagenum{\ifnum#1=1\bfseries\fi#7}}\par
    \nobreak
    \endgroup
  \fi}
\AtBeginDocument{%
\expandafter\renewcommand\csname r@tocindent0\endcsname{0pt}
}
\def\l@subsection{\@tocline{2}{0pt}{2.5pc}{5pc}{}}
\def\l@subsubsection{\@tocline{2}{0pt}{3.5pc}{5pc}{}}

\makeatother

\begin{document}



\title[Density estimation for censored circular data]{Nonparametric optimal density estimation for censored circular data}

\author[N. Conanec, C. Lacour, T. M. Pham Ngoc]{Nicolas CONANEC$^{1*}$, Claire LACOUR$^2$ and Thanh Mai PHAM NGOC$^1$}

\address{$^{1}$  LAGA (UMR 7539), Universit\'e Sorbonne Paris Nord, Institut Galil\'ee, Villetaneuse, France.}
\email{\textcolor[rgb]{0.00,0.00,0.84}{\{conanec, phamngoc\}@math.univ-paris13.fr}}

\address{$^{2}$ Univ Gustave Eiffel, Univ Paris Est Creteil, CNRS, LAMA UMR8050 F-77447 Marne-la-Vallée, France.}
\email{\textcolor[rgb]{0.00,0.00,0.84}{claire.lacour@univ-eiffel.fr}}

\begin{abstract}
We consider the problem of estimating the probability density function of a circular random variable observed under censoring. To this end, we introduce a projection estimator constructed via a regression  approach on linear sieves. We first establish a lower bound for the mean integrated squared error in the case of Sobolev densities, thereby identifying the minimax rate of convergence for this estimation problem. We then derive a matching upper bound for the same risk, showing that the proposed estimator attains the minimax rate when the underlying density belongs to a Sobolev class. Finally, we develop a data-driven version of the procedure that preserves this optimal rate, thus yielding an adaptive estimator. The practical performance of the method is demonstrated through simulation studies.
\end{abstract} 

\maketitle

\tableofcontents


\section{Introduction}\setcounter{equation}{0}

\vphantom{\cite{MardiaJupp} \cite{Jammalamadaka-Book} \cite{Kaplan-Meier} \cite{Turnbull-1974} \cite{Jammalamadaka-2009} \cite{Conanec-2025} \cite{Comte-Genon-2020} \cite{Baudry-2012} \cite{Baraud-2002} \cite{Talagrand} \cite{Klein-Rio} \cite{Tsybakov} \cite{Ley-Verdebout} \cite{Diamond-McDonald-1992} \cite{Wellner-1995} \cite{vanDerLaan-Jewell-1995} \cite{Bhattacharyya-1985} \cite{Efromovich-2021} \cite{Alotaibi-2025} \cite{Pewsey-GarciaPortuge-2021}  }

Directional statistics concerns the analysis of random variables taking values as directions rather than on the real line.  In this paper we focus on circular data which are encountered in various scientific fields, such as biology (directions of animal migration), bioinformatics (protein conformational angles), geology (rock fracture orientations), medicine (circadian rhythms), forensics (crime timing), and the social sciences (time-of-day or calendar effects). Comprehensive surveys of statistical methods for circular data can be found in \citelink{MardiaJupp}{Mardia \& Jupp (2000)}, \citelink{Jammalamadaka-Book}{Jammalamadaka \& Sengupta (2001)}, \citelink{Ley-Verdebout}{Ley \& Verdebout (2017)}, and recent advances are compiled in \citelink{Pewsey-GarciaPortuge-2021}{Pewsey and Garc\'ia-Portugu\'es  (2021)}.

In many statistical applications, full observations are unavailable due to censoring, where only partial information about the variable of interest is observed. Censoring problems have been extensively studied on the real line, and several censoring schemes have been proposed, each providing a distinct level of information. 
Classical contributions include the Kaplan-Meier estimator (\citelink{Kaplan-Meier}{Kaplan \& Meier ('58)}), which provides a nonparametric maximum likelihood estimator of the survival function under right censoring. Subsequent developments have addressed more complex censoring structures, including doubly censored  data (\citelink{Turnbull-1974}{Turnbull ('74)}), the current status model (\citelink{Diamond-McDonald-1992}{Diamond \& McDonald ('92)}) and interval censoring case 2  (\citelink{Wellner-1995}{Wellner ('95)}). Further extensions considered Type II censoring (\citelink{Bhattacharyya-1985}{Bhattacharyya ('85)} and its modern variants \citelink{Alotaibi-2025}{Alotaibi et al. (2025)}). Recent advances include sharp adaptive minimax results for the current status setting (\citelink{Efromovich-2021}{Efromovich (2021)}) and extensions to doubly censored models (\citelink{vanDerLaan-Jewell-1995}{van der Laan \& Jewell ('95)}).

When considering censored estimation problems on the circle, most censoring mechanisms defined on the real line cannot be transposed in a mathematically consistent way. To the best of our knowledge, the only well-posed circular censoring model available in the literature is that introduced by \citelink{Jammalamadaka-2009}{Jammalamadaka \& Mangalam (2009)} in the context of cumulative distribution function estimation. The present work aims to bridge this challenging gap by developing a framework for nonparametric density estimation of circular data subject to censoring.

This paper substantially extends and refines the contribution of \citelink{Conanec-2025}{Conanec (2025)}  who introduced an estimator of the probability density function of a circular random variable observed under censoring. We first establish the minimax lower bound for the mean integrated squared error (MISE) in this circular censoring framework. We then propose a new estimation method that directly targets the density $f$ and achieves both minimax optimality and adaptivity. In contrast to \citelink{Conanec-2025}{Conanec (2025)}, where $f$ was constructed as a quotient estimator entailing theoretical and practical drawbacks, our approach is based on a direct projection method that overcomes these limitations. To this end, we introduce a least-squares type contrast function $\zeta$ defined via a new scalar product, namely the Euclidean inner product on $\cercle$, weighted by a function $\sigma$ that reflects the censoring mechanism. The projection estimator on the trigonometric space $S_m$ with dimension $D_m=2m+1$ is then given by
 $\mathring{f}_m =  \argmin{t\in S_m} \hspace{3pt} \zeta(t)$ and admits a fully explicit expression.  We then show that the MISE of our estimator attains the minimax rate when $f$ belongs to a Sobolev class, proving its optimality. To handle the realistic case where the smoothness of $f$ is unknown, we design a data-driven selection procedure based on a penalized criterion on $m$. The resulting adaptive estimator $\hat{f}_{\hat{m}}$ satisfies an oracle inequality and preserves the optimal convergence rate. Thus, our method provides an estimator that is both theoretically optimal and fully adaptive. Finally, we complement our theoretical results with numerical experiments that demonstrate the effectiveness and robustness of this new  approach.

The article is organized as follows. In Section~\ref{SectionModele}, we describe the specific features of the circular framework and discuss how the censoring mechanism affects the available information. Section~\ref{SectionBorneInf} establishes the minimax lower bound for the considered estimation problem. In Section~\ref{SectionResultats}, we construct the proposed estimator, derive its rate of convergence in a classical setting, and present a data-driven selection procedure that yields an adaptive version without deteriorating the convergence rate. Section~\ref{SectionSimulations} reports numerical experiments illustrating the performance of our estimator under various scenarios, and all proofs are collected in Section~\ref{sec:proofs}.
\\

\section{Definition of the model}\setcounter{equation}{0}\label{SectionModele}

\subsection{Circular context}

We begin by addressing the specific challenges arising from the circular nature of the data. In particular, since we work under an interval censoring model, a natural question is how to define whether a point belongs to an interval on the circle. More generally, can the usual partial order on $\mathbb{R}$ be meaningfully extended to the unit circle $\mathbb{S}^1$? Without additional structure, such an extension is not possible: for instance, it is meaningless to decide whether $\frac{\pi}{2}$ is greater than $\frac{3\pi}{2}$. To resolve this ambiguity, we adopt the following convention. We fix an initial direction, denoted by $0$, and choose an orientation on the circle.  We choose for the rest of the article the anticlockwise orientation.
Each angle is then represented by its unique equivalent in $[0, 2\pi)$ modulo $2\pi$. Under these conventions, the circle $\cercle$ can be represented by the interval $[0, 2\pi)$, and each angle corresponds to a point within that interval. We then use the standard order on $[0, 2\pi)$ as our circular order.

A further question arises when defining intervals between two points on $\mathbb{S}^1$. On the real line, two points determine a unique interval; on the circle, however, they define two complementary arcs whose union covers the entire circle.  Accordingly, throughout the paper, when we write an interval $[x, y]$ on the circle, we refer to the set of points obtained by moving from $x$ to $y$ along the chosen orientation. In other words, for any 
$(x, y)\in\left(\mathbb{S}^1\right)^2, \mathbb{S}^1 = [x, y)\sqcup [y, x)$.
Thus if we have $x > y$ two points of the interval, $[x, y]$ is the set $[0, 2\pi) \backslash (y, x)$, i.e $[x, y] = [x, 2\pi)\sqcup [0, y]$, and $(y, x)$ is the classic interval as we know it. 
In Figure~\ref{fig:intervalle_circulaire_cercle_et_r} we can see two circles with intervals $\left[ \frac{\pi}{3}, \frac{\pi}{6} \right]$ and $\left[\frac{\pi}{6}, \frac{\pi}{3} \right]$ highlighted and the same intervals on the real line.

\begin{figure}[ht]

\begin{center}
\begin{tikzpicture}[scale=1.5]

  \begin{scope}[xshift=-2cm, yshift=1.6cm]
    \draw[red, thick] (0,0) circle(1);
    
    \foreach \angle/\label in {
        0/{$0$},
        90/{$\frac{\pi}{2}$},
        180/{$\pi$},
        270/{$\frac{3\pi}{2}$}
    } {
      \draw (\angle:1.2) node {\label};
      \draw (-1.05,0) -- (-0.95,0);
      \draw (1.05,0) -- (0.95,0);
      \draw (0,-1.05) -- (0,-0.95);
      \draw (0,1.05) -- (0,0.95);
    }

    \node at (0,0) {+};

    \draw[black, thick] (20:1) arc (20:60:1);
    \fill[red] (20:1) circle(0.025);
    \fill[red] (60:1) circle(0.025);
    
    \foreach \a in {70, 80,...,350} {
      \draw[red, thick] (\a:1) node[cross=2.1pt, rotate=\a] {};
    }
    \foreach \a in {0, 10, 20} {
      \draw[red, thick] (\a:1) node[cross=2.1pt, rotate=\a] {};
    }
  \end{scope}

  \begin{scope}[xshift=-2cm, yshift=-1.6cm]
    \draw[black, thick] (0,0) circle(1);
    
    \foreach \angle/\label in {
        0/{$0$},
        90/{$\frac{\pi}{2}$},
        180/{$\pi$},
        270/{$\frac{3\pi}{2}$}
    } {
      \draw (\angle:1.2) node {\label};
      \draw (-1.05,0) -- (-0.95,0);
      \draw (1.05,0) -- (0.95,0);
      \draw (0,-1.05) -- (0,-0.95);
      \draw (0,1.05) -- (0,0.95);
    }

    \node at (0,0) {+};

    \draw[red, thick] (20:1) arc (20:60:1);
    \fill[red] (20:1) circle(0.025);
    \fill[red] (60:1) circle(0.025);
    
    \foreach \a in {30, 40, 50} {
      \draw[red, thick] (\a:1) node[cross=2.1pt, rotate=\a] {};
    }
  \end{scope}
  
  \begin{scope}[xshift=2cm, yshift=1.6cm]
  
  \draw[red, thick] (0,0) -- (3.14,0);
  
   \foreach \angle/\label in {
        {0*3.14/360}/{$0$},
        {90*3.14/360}/{$\frac{\pi}{2}$},
        {180*3.14/360}/{$\pi$},
        {270*3.14/360}/{$\frac{3\pi}{2}$},
        {360*3.14/360}/{$2\pi$}
    } {
      \draw (\angle,0.5) node {\label};
      \draw (\angle,0.05) -- (\angle,-0.05);
    }
    
    
    \draw[black, thick] (20*3.14/360,0) -- (60*3.14/360,0);
    \fill[red] (20*3.14/360,0) circle(0.025);
    \fill[red] (60*3.14/360,0) circle(0.025);
    
    \foreach \a in {70,80,...,360} {
    \pgfmathsetmacro{\x}{\a*3.14/360}
      \draw[red, thick] (\x,0) node[cross=1.7pt] {};
    }
    \foreach \a in {0, 10} {
    \pgfmathsetmacro{\x}{\a*3.14/360}
      \draw[red, thick] (\x,0) node[cross=1.7pt] {};
    }
    
  \end{scope}
  
   \begin{scope}[xshift=2cm, yshift=-1.6cm]
  
  \draw[black, thick] (0,0) -- (3.14,0);
  
   \foreach \angle/\label in {
        {0*3.14/360}/{$0$},
        {90*3.14/360}/{$\frac{\pi}{2}$},
        {180*3.14/360}/{$\pi$},
        {270*3.14/360}/{$\frac{3\pi}{2}$},
        {360*3.14/360}/{$2\pi$}
    } {
      \draw (\angle,0.5) node {\label};
      \draw (\angle,0.05) -- (\angle,-0.05);
    }
    
    
    \draw[red, thick] (20*3.14/360,0) -- (60*3.14/360,0);
    \fill[red] (30*3.14/360,0) circle(0.025);
    \fill[red] (60*3.14/360,0) circle(0.025);
    
    \foreach \a in {30, 40, 50} {
    \pgfmathsetmacro{\x}{\a*3.14/360}
      \draw[red, thick] (\x,0) node[cross=1.7pt] {};
    }
    
  \end{scope}

\end{tikzpicture}
\end{center}
\caption{On the top left is the circular interval $\left[ \frac{\pi}{3}, \frac{\pi}{6} \right]$ on $\cercle$ and on the top right on  $\mathbb{R}$. On the bottom left is the circular interval $\left[ \frac{\pi}{6}, \frac{\pi}{3} \right]$ on $\cercle$ and on the bottom right on $\mathbb{R}$.}
  \label{fig:intervalle_circulaire_cercle_et_r}
\end{figure}
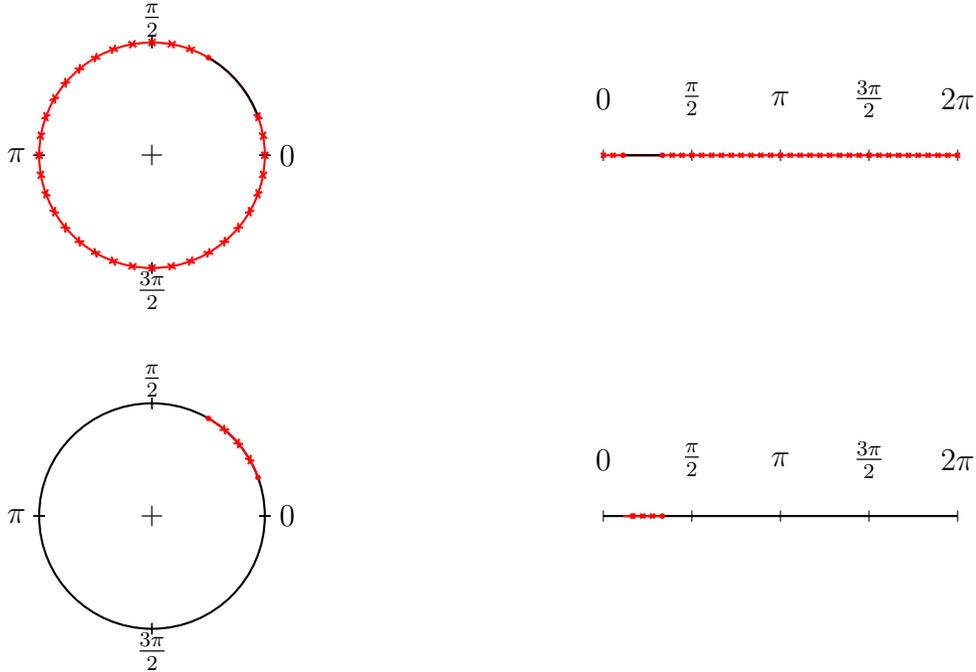

\subsection{The circular censor model}

Set $(\Omega, \mathcal{F},\Proba)$ a probability space. We consider functions in $\mathbb{L}^2(\cercle)$, endowed with the usual scalar product $<g,h>_2 = \Inte g(x) h(x) \dd x$ and the associated norm will be written $\|\cdot\|_2$.
Let $(X_1, L_1, U_1),$ $ \dots , (X_n, L_n, U_n)$ be a sample of the triplet $(X, L, U)$ where $(L, U)$ are the censor elements and $X$ is the variable of interest, and the sequences $(L_i, U_i)_{1\leq i \leq n}$ and $(X_i)_{1\leq i \leq n}$ are independent. 
The censoring mechanism operates as follows:  we observe $X_i$ exactly when $X_i \in [L_i, U_i]$, and otherwise $X_i$ is unobserved, in which case we only retain the information contained in the couple $(L_i, U_i)$. Because of this structure, we refer to $[L_i, U_i]$ as the window of observation, and to $(U_i, L_i)$ as the censoring arc.
In the case of a censored observation, we assign the observation to an arbitrary point outside the interval $[0, 2\pi]$; specifically, we set it to $-\pi$.  We suppose that 
\[L \ne U \hspace{2pt} \textit{a.s.}, \]
which guarantees that, almost surely, the variable $X$ may fall within its observation window.
Remark that we only need $L$ and $U$ to be different and not ordered. This is important as $[L,U]$ and $[U,L]$ do not define the same intervals.
We then define
\begin{equation}\nonumber
\left\{
\begin{array}{ll}
X_i'=\left\{
\begin{array}{ll}
        X_i \text{ , if } X_i \in [L_i, U_i] ,\\
        -\pi \text{ , otherwise.}
    \end{array}
\right.\\
\Delta_i= \mathds{1}_{\{X_i\in[L_i, U_i]\}}.
\end{array}
\right.
\end{equation}
So our observations are the sample of triplets $(X_1', L_1, U_1),\dots, (X_n', L_n, U_n)$ and the aim is to build an estimator of the density function $f$ of the circular variable $X$.

\section{Lower bound}\setcounter{equation}{0}\label{SectionBorneInf}

We first establish the minimax risk lower bound associated with this  censored circular density estimation problem. To that end, we consider smoothness classes characterized by Sobolev regularity.
We recall the definition of the Sobolev class $W(\beta, R)$. We define $\alpha_j$ as the following
\begin{equation}\nonumber
\alpha_j=\left\{\begin{array}{ll}
        j^{\beta}, \hspace{25pt}\text{for even }j,\\
        (j+1)^{\beta}, \hspace{0pt}\text{for odd }j.
    \end{array}
\right.
\end{equation}
We can define the Sobolev class $W(\beta, R)$ for $\beta>0$ and $R>0$ as the following set of functions
\begin{equation}\label{Sobolev_Def}
W(\beta,R) = \left\{ g \in \mathbb{L}^2(\cercle), \sum_{j=0}^{+\infty} (\alpha_j^2 |<g,\varphi_j>_2|^2) \hspace{3pt}< \frac{R^2}{\pi^{2\beta}} \right\},
\end{equation}
where $\{ \varphi_j\}_{j\in\mathbb{N}}$ is the trigonometric basis of $\mathbb{L}^2(\cercle)$, i.e $ \varphi_0=\frac{1}{\sqrt{2\pi}}$, and for $j\in\mathbb{N}^*, \varphi_{2j-1} = \frac{1}{\sqrt{\pi}}\cos(j\cdot), \varphi_{2j}= \frac{1}{\sqrt{\pi}}\sin(j\cdot)$ .\\

We now present the lower bound on the minimax risk.
\begin{theorem}\label{BorneInf}
Assume $\beta \geq \frac{1}{2}$. The estimation problem we introduced has the following lower bound 

\[ \liminf_{n \rightarrow +\infty} \inf_{\hat{f}}\sup_{f\in W(\beta,R)} \Esp_f\left( n^{\frac{2\beta}{2\beta +1}}\| \hat{f} - f\|_2^2\right) \geq c ,\]
where $\inf_{\hat{f}}$ denotes the infimum over all estimators defined using the sample $(X_1', L_1, U_1$, $\dots,$ $X_n', L_n, U_n)$ and where the constant $c>0$ depends only on $\beta, R$.
\end{theorem}
We can find the proof of this theorem in Section~\ref{LowerBound}.\\

The rate of convergence obtained in Theorem \ref{BorneInf}  corresponds to the optimal convergence rate for estimating a univariate probability density belonging to a Sobolev class of smoothness $\beta$. Interestingly, the censoring mechanism does not affect the rate of convergence, which is a satisfactory feature. A noteworthy observation is that the lower bound itself depends on the regularity of $f$. Consequently, the estimator proposed in \citelink{Conanec-2025}{Conanec (2025)} is not minimax optimal unless $f$ and $f\sigma$ share the same smoothness, where $\sigma$ is defined in \eqref{eq:def_sigma}. Designing an estimation procedure that achieves minimax optimality remains a challenging problem and constitutes the focus of the following section.

%

\section{Estimator and results}\setcounter{equation}{0}\label{SectionResultats}

\subsection{Estimator procedure}

To define the estimation procedure, we first define the following function
\begin{equation}\label{eq:def_sigma}
\sigma : x\in\mathbb{S}^1\mapsto \Proba (x\in [L, U]) =\Inte\Inte \mathds{1}_{\{x\in[l, u]\}} \Proba_{(L, U)}(\dd l, \dd u).
\end{equation}
We make the following assumption on our model.\\
\assumption{(A)}{
There exists a real $\sigma_0>0$ such that, for all $x$ in $\cercle$,\[ \sigma(x)\geq \sigma_0 >0.\]}
This assumption ensures that, with nonzero probability, any point on the circle can fall within an observation window. It provides a theoretical guarantee that every point $x \in \cercle$ is observable and, given a sufficiently large sample, can be estimated by our procedure. Although $\sigma_0$ is an unknown quantity and cannot be used in practice for estimation, it remains useful for theoretical analysis.

We observe that, for any integrable function $t$, 
\begin{align}
\nonumber
\Esp(\Delta_i t(X_i'))&= \Esp(\Delta_i t(X_i))\\ 
\nonumber
&= \Inte\Inte\Inte \mathds{1}_{\{x\in[l, u]\}} t(x) f(x) \Proba_{(L, U)}(\dd l, \dd u) \dd x \\
\nonumber
&= \Inte t(x) f(x) \underbrace{ \Inte\Inte \mathds{1}_{\{x\in[l, u]\}} \Proba_{(L, U)}(\dd l, \dd u)}_{=\sigma(x)} \dd x \\
\nonumber
&= \Inte t(x) f(x)\sigma(x) \dd x\\
\label{eq:Esperance_et_psi}
&= <t, f\sigma>_2 ,
\end{align}
This equation provides a way to construct a new scalar product on $\mathbb{L}^2(\cercle)$, weighted by the function $\sigma$. We write $$<g, h>_{2,\sigma} = \Inte g(x)h(x)\sigma(x) \dd x$$ for any $g, h\in\mathbb{L}^2(\cercle)$. It defines a scalar product on $\mathbb{L}^2(\cercle)$ and we write  $\|\cdot\|_{2,\sigma}$ the norm associated. 
Since $\sigma$ is defined as a probability we consider the following empirical estimator of $\sigma$. We set
\begin{equation}\label{eq:def_hat_sigma}
\hat{\sigma}: x\in\cercle \mapsto \frac{1}{n} \sum_{i=1}^n \mathds{1}_{\{x\in[L_i, U_i]\}}. 
\end{equation}
From this we can define the empirical counterpart of $\|\cdot\|_{2,\sigma}$ and $<\cdot,\cdot>_{2,\sigma}$
We define $$<g,h>_{n,\sigma} = \frac{1}{n} \sum_{i=1}^n \int_{[L_i,U_i]} g(x)h(x) \dd x = \Inte g(x)h(x)\hat{\sigma}(x) \dd x ,$$ and $$\|g\|_{n,\sigma}^2= \Inte g^2(x) \hat{\sigma}(x) \dd x$$ for any $g,h\in\mathbb{L}^2(\cercle)$. The bilinear symmetric positive semi-definite form $<\cdot,\cdot>_{n,\sigma}$ is not a scalar product as $<f,f>_{n,\sigma}$ only guarantees that $f(x)=0$ for $x\in \bigcup_{i=1}^n [L_i,U_i]$, not on all $\cercle$. The same argument tells us that $\|\cdot\|_{n,\sigma}$ is not a norm but a semi-norm.
The bilinear form $<\cdot,\cdot>_{n,\sigma}$ can be seen as an empirical counterpart of the scalar product $<\cdot,\cdot>_{2,\sigma}$ because, for any $g, h\in\mathbb{L}^2(\cercle)$,
\[ \Esp(<g, h>_{n, \sigma}) = <g, h>_{2, \sigma}.\]
To estimate $f$ we define the following empirical contrast function
\[ \zeta(t) := \frac{1}{n}\sum_{i=1}^n \left(\Inte t^2(x)\mathds{1}_{x\in[L_i, U_i]} \dd x - 2\Delta_i t(X_i') \right),\]
for any function $t$ on $\cercle$.
This choice is motivated by the following computation. If we take the expectation of $\zeta(t)$ for a function $t$ we have 
\begin{align}
\nonumber\Esp(\zeta(t)) &= \Esp\left( \Inte t^2(x)\mathds{1}_{x\in[L, U]} \dd x \right) -2\Esp(\Delta t(X'))\\
\nonumber&=\Inte t^2(x)\sigma(x) \dd x -2 \Inte t(x)f(x)\sigma(x) \dd x \\
\label{eq:justif_contraste}&= \|t-f\|_{2, \sigma}^2 - \|f\|_{2, \sigma}^2.
\end{align}
This contrast can be used to define a new estimator for $f$ taking its minimum over a subspace $S_m$ of $\mathbb{L}^2(\cercle)$. 

The subspaces we consider are the trigonometric spaces that we recall now. 
Set $(S_m)_{m\in\mathcal{M}_n}$ to be a collection of linear spaces where $\mathcal{M}_n = \{ 1, \dots,  \lfloor n/2 \rfloor-1\}$ is the set of possible values of $m$.  Each $S_m$ is the subspace generated by $\{\varphi_0=\frac{1}{\sqrt{2\pi}}, \varphi_{2j-1}=\frac{1}{\sqrt{\pi}}\cos(j\cdot), \varphi_{2j}=\frac{1}{\sqrt{\pi}}\sin(j\cdot) | \text{ for } j \in \{ 1,\dots , m \}\}$. Its dimension is $D_m=2m+1$.
Taking $\Phi_0 =\frac{1}{\sqrt{2\pi}}$  the following inequality is easily obtained:
\begin{equation}\label{eq:Linear_Sieves}
\forall m \in \mathcal{M}_n, \forall t\in S_m, \|t\|_{\infty} \leq \Phi_0 \sqrt{D_m} \|t\|_2.
\end{equation}
Moreover, for $m\in\mathcal{M}_n$, if $(\philambda)_{\indices}$, where $|\Lambda_m|=D_m$, is an orthonormal basis of $S_m$, an equivalent version of \eqref{eq:Linear_Sieves} is
\begin{equation}\label{eq:Linear_Sieves2}
\forall m \in \mathcal{M}_n, \|\sum_{\indices} \philambda^2\|_{\infty} \leq \Phi_0^2 D_m.
\end{equation}
We now define a function of interest. We denote by ${}^t M $ the transpose of a vector or  a matrix $M$. For $m\in\mathcal{M}_n$
\begin{equation}\label{eq:estimateur_argmin}
 \mathring{f}_m = \argmin{t\in S_m} \hspace{3pt} \zeta(t) = {}^t \hat{A}_m \overrightarrow{\varphi_m},
 \end{equation}
where $\hat{A}_m := (\hat{a}_{\lambda})_{\indices}$ are the Fourier coefficients of  $\mathring{f}_m$, that is the coefficients of $\mathring{f}_m$ in the trigonometric basis $\overrightarrow{\varphi_m} = (\philambda)_{\indices}$. Thus, thanks to \eqref{eq:justif_contraste} the function $\mathring{f}_m$ that minimizes the contrast $\zeta$ on $S_m$ is likely to minimize on $S_m$ the norm $\|t-f\|_{2, \sigma}^2$ and thus to be a relevant estimator of $f$.\\
Since $S_m$ is a finite dimension subspace of $\mathbb{L}^2(\cercle)$, we can represent a scalar product on this subspace with a matrix. For $(\philambda)_{\indices}$ the trigonometric basis of $S_m$ for the norm $\|\cdot\|_2$ we write $G_m$ the Gram matrix for the basis $(\philambda)_{\indices}$ and the scalar product $<\cdot, \cdot>_{2, \sigma}$ and $\hat{G}_m$ the empirical counterpart of $G_m$. In fact, $G_m$ and $\hat{G}_m$ are matrices of dimension $D_m \times D_m$ and for $(\lambda, \lambda')\in \Lambda_m^2$
\begin{equation}\label{eq:Def_G}
(G_m)_{\lambda, \lambda' \in \Lambda_m} = <\philambda, \varphi_{\lambda'}>_{2, \sigma} = \Inte \philambda(x)\varphi_{\lambda'}(x) \sigma(x) \dd x.
\end{equation}
\begin{equation}\label{eq:Def_G_chapeau}
 (\hat{G}_m)_{\lambda, \lambda' \in \Lambda_m} = <\philambda, \varphi_{\lambda'}>_{n, \sigma} = \frac{1}{n}\sum_{i=1}^n\ \int_{[L_i, U_i]} \philambda(x)\varphi_{\lambda'}(x) \dd x.
 \end{equation}
The following lemma shows that we can compute the value of the coefficients of $\mathring{f}_m$ using matrix $\hat G_m$. Note the closeness of our problem to a regression issue, with the design empirical measure replaced here by $\hat \sigma(x)\dd x$.
\begin{lemma}\label{def_coefs}
Assume $\hat{G}^{-1}_m$ is invertible. Set $\hat{U}_m = \left(\frac{1}{n}\sum_{i=1}^n \Delta_i \philambda(X_i)\right)_{\indices}$. Then the coefficients of $\mathring{f}_m$ are defined as 
\begin{equation}\label{eq:def_des_coefs}
\hat{A}_m= \hat{G}_m^{-1}\hat{U}_m.
\end{equation}
\end{lemma}
 The proof can be found in Section~\ref{preuve_def_coefs}.\\
Hence, under invertibility conditions,  the function $\mathring{f}_m$ can be computed. The next result shows which conditions ensure the invertibility of $G_m$ and $\hat{G}_m$. 
\begin{lemma}\label{inversible}
For all $m$ in $\mathcal{M}_n $, $\hat{G}_m$ is invertible a.s. and $G_m$ is invertible if Assumption (A) is true.
\end{lemma}
The proof can be found in Section~\ref{inversion}.\\
The function $\mathring{f}_m$ is now well defined. Finally, in the same way as \citelink{Baraud-2002}{Baraud (2002)}, the estimator we consider is the following
\begin{equation} \label{eq:Definition_f_chapeau}
\tilde{f}_m : x\in\cercle\mapsto \left\{
\begin{array}{ll}
        \mathring{f}_m (x) \text{, if } \|\mathring{f}_m\|_2^2 \leq k_n, \\
        0 \text{, otherwise},
    \end{array}
\right.
\end{equation}
where we will take $k_n = n^2$ which turns to be a  convenient value in the proofs.

\subsection{MISE upper bound}

In this section we study the theoretical properties of  the estimator    $\tilde{f}_m$ defined in \eqref{eq:Definition_f_chapeau}.  
We denote by $\|f\|_{\infty}$ the essential supremum norm of $f$ and by $\mathbb{L}^{\infty}(\cercle)$  the vector space of essentially bounded measurable functions on the circle.
The next result explicits an upper bound for the MISE.

\begin{theorem}\label{MISE du deuxieme estimateur}
Suppose that $f\in\mathbb{L}^{\infty}(\cercle)$ and Assumption (A).
Then an upper bound of $\tilde{f}_m$ MISE is

\[\Esp(\|\tilde{f}_m - f\|^2_2) \leq C_1\|f_m-f\|_2^2 + C_2\frac{D_m}{n} +C_3\frac{1}{n} ,\]
where $f_m$ is the projection of $f$ on $S_m$ and $C_1$ and $C_2$ are constants that depend only on $\sigma_0$ and $C_3$ depend only on $\sigma_0$ and $\|f\|_\infty$.\\
In particular if $f$ is an element of $W(\beta,R)$ and if we choose $m=m_n$ such that $D_{m_n} = \lfloor n^{\frac{1}{2\beta +1}}\rfloor$ we have the following rate

\[\Esp(\|\tilde{f}_m -f\|_2^2) =  \mathcal{O}\left( n^{\frac{-2\beta}{2\beta +1}} \right).\]
\end{theorem}
The proof can be found in Section~\ref{Preuve_MISE_Deuxieme_Estimateur}. \\
The convergence rate established in Theorem \ref{MISE du deuxieme estimateur} is in fact sharper than that obtained in \citelink{Conanec-2025}{Conanec (2025)}, as it depends on the smoothness of $f$ itself rather than that of $f\sigma$, which may differ. Combined with the lower bound given in Theorem~\ref{BorneInf}, this result shows that $\tilde{f}_m$ is a minimax optimal estimator of $f$.

\subsection{Adaptive estimation procedure}

We now aim to derive from this estimation procedure an adaptive estimator. Indeed, Proposition~\ref{MISE du deuxieme estimateur} shows that the estimator $\tilde{f}_m$ achieves the optimal convergence rate, but only when the dimension parameter $D_m$ is chosen appropriately, that is, when the smoothness $\beta$ of $f$ is known. Since $\beta$ is typically unknown, we propose a data-driven procedure to automatically select the value of $m$ that yields the best trade-off.
\\

Next  Proposition \ref{Adaptation du deuxieme estimateur}  proposes a data-driven estimator which verifies an oracle-type inequality. We recall that the operator norm of the matrix $M$, $\|M\|_{op}$, is the square root of the greatest eigenvalue of the matrix ${}^t MM$. 

\begin{proposition}\label{Adaptation du deuxieme estimateur}
Suppose that $f\in\mathbb{L}^{\infty}(\cercle)$ and Assumption (A) .  For $\hat{m}$ defined as 
\[ \hat{m} = \argmin{m\in \mathcal{M}_n}\left( \zeta (\mathring{f}_m) + {\rm{pen}}(m)\right),\]
where $\rm{pen}(\cdot)$ that verifies
\[ {\rm{pen}}(m) \geq \kappa \frac{\|G_m^{-1}\|_{op}}{2\pi}\frac{D_m}{n}\Inte f(x)\sigma(x) \dd x,\]
with $\kappa$ universal constant ($\kappa > 16$ works), we obtain the following inequality for the MISE of $\tilde{f}_{\hat{m}}$
\[ \Esp(\|\tilde{f}_{\hat{m}} - f\|^2_2) \leq C(\sigma_0) \inf_{m\in\mathcal{M}_n}(\|f-f_m\|_2^2 + {\rm{pen}}(m)) + \frac{\tilde{C}(\sigma_0, \|f\|_{\infty})}{n},\]
where $f_m$ is the projection of $f$ on $S_m$.
\end{proposition}
The proof can be found in Section~\ref{Preuve_Adaptation_Deuxieme_Estimateur}.

The estimator $\tilde{f}_{\hat{m}}$ satisfies an oracle-type inequality, but the penalty term involves two unknown constants, namely $\|G_m^{-1}\|_{op}$ and $\int f(x)\sigma(x)\dd x$. To obtain a fully computable penalty, we proceed as follows. We note that $\int f(x)\sigma(x) \dd x = \mathbb{E}(\Delta) \leq 1$, so this quantity does not need to be estimated. The other term $\|G_m^{-1}\|_{op}$ only depends on the censoring and can be seen as the norm of a precision matrix.
The importance of this inverse matrix is related to \citelink{Comte-Genon-2020}{Comte \& Genon-Catalot (2020)}, who have studied non-parametric regression as a partly inverse problem. 
To estimate it we use its empirical counterpart $\|\hat G_m^{-1}\|_{op}$, which can be shown to be a strongly consistent estimator. The following theorem establishes that, under this substitution, the resulting estimator with a fully computable penalty still satisfies the same type of oracle inequality. Note that alternative approaches exist to handle these constants, but they typically modify the penalty term by using rough bounds and introducing additional quantities such as $\| f\| _\infty$ or $\sigma_0$, which are also unknown, making them less satisfactory in practice.

\begin{theorem}\label{Adaptation sans inconnues du deuxieme estimateur}
Suppose that $f\in\mathbb{L}^{\infty}(\cercle)$ and Assumption \textrm{(A)}.  For $\hat{m}^*$ defined as 
\[ \hat{m}^* = \argmin{m\in \mathcal{M}_n}\left( \zeta (\mathring{f}_m) + \widehat{\rm{pen}}(m)\right),\]
and
\[ \widehat{\rm{pen}}(m) = \kappa \frac{\|\hat{G}_m^{-1}\|_{op}}{2\pi}\frac{D_m}{n},\]
with $\kappa$ universal constant ($\kappa > 32$ works), we obtain the following inequality for the MISE of $\tilde{f}_{\hat{m}^*}$
\[ \Esp(\|\tilde{f}_{\hat{m}^*} - f\|^2_2) \leq K(\sigma_0) \inf_{m\in\mathcal{M}_n}\left(\|f-f_m\|_2^2 + \kappa\frac{D_m}{n}\right) + \frac{\tilde{K}(\sigma_0, \|f\|_{\infty})}{n},\]
where $f_m$ is the projection of $f$ on $S_m$.
\end{theorem}
The proof can be found in Section~\ref{Preuve_Adaptation_Sans_Inconnues_Deuxieme_Estimateur}.\\
The following corollary establishes the convergence rate of the estimator $\tilde{f}_{\hat{m}^*}$ when $f$ belongs to a Sobolev class.

\begin{corollary}\label{Coro deuxieme estimateur}
Moreover if we assume that $f$ is an element of $W(\beta,R)$ with $\beta>0$, we obtain
\[\Esp(\|\tilde{f}_{\hat{m}^*} -f\|_2^2) =  \mathcal{O}\left( n^{\frac{-2\beta}{2\beta +1}} \right).\]
\end{corollary}
This corollary can be proved using the oracle inequality obtained in Theorem~\ref{Adaptation sans inconnues du deuxieme estimateur} and similar final computations of the proof of Theorem ~\ref{MISE du deuxieme estimateur}.
It tells us that the estimator $\tilde{f}_{\hat{m}^*}$ of $f$ is adaptive and minimax optimal.

\section{Simulations}\setcounter{equation}{0}\label{SectionSimulations}

We numerically implement the estimation procedure of Theorem~\ref{Adaptation sans inconnues du deuxieme estimateur}. To this end,  we recall the necessary quantities for our estimation procedure. The set of possible models is represented by $\mathcal{M}_n = \{ 1,\cdots, \lfloor n/2 \rfloor -1 \}$. In practice for large values of $n$ we know that we do not have to test all possible models  since the contrast penalizes complex models. Consequently,  we use $\mathcal{M}_n = \{ 1,\cdots, \max(\lfloor n/2 \rfloor -1,25) \}$. We then compute for all $m$ in $\mathcal{M}_n$ the penalty term evaluated in $m$ and the contrast evaluated in  $\mathring{f}_m$. For the contrast term, using \eqref{eq:Contraste}, and \eqref{eq:def_des_coefs} we have 
\[\zeta(\mathring{f}_m) = -{}^t \hat{A}_m \hat{U}_m.\]
To compute the penalty term we use the package CAPUSHE (see \citelink{Baudry}{Baudry (2012)}) that calibrates the tuning constant $\kappa$ via the slope heuristics. The selected model $\hat{m}^*$ is then defined as:
\[ \hat{m}^* = \argmin{m\in \mathcal{M}_n}\left( -\sum_{\indices} \frac{\hat{a}_\lambda}{n}\left[\sum_{i=1}^n \Delta_i\philambda(X_i) \right] + \kappa\frac{2m +1}{2\pi n}\|\hat{G}_m^{-1}\|_{op}\right).\]
Finally the estimator of $f$  is given by $\tilde{f}_{\hat{m}^*}= \left(\sum_{\lambda\in \Lambda_{\hat{m}^*}} \hat{a}_\lambda \philambda\right)\mathds{1}_{\{\sum_{\lambda\in \Lambda_{\hat{m}^*}}\hat{a}_\lambda^2 \leq n^2\}}$.\\

To assess the performance of our estimator $\tilde{f}_{\hat{m}^*}$, we consider the problem of estimating the probability density function of a Von Mises distribution $M(\mu,k)$ which is given by
\[ f(x) = \frac{1}{2\pi I_0(k)}e^{k \cos(x-\mu)},\]
where $\mu$ is the mean direction, $k$ is the concentration parameter and $I_0$ is the modified Bessel function of the first kind of order $0$ such that $f$ is a density on $\cercle$. This distribution is the circular equivalent of the Gaussian distribution on the real line. We consider different types of censorship and use a Monte Carlo method with $N=100$ samples of different sizes to estimate the MISE of our estimator. The four models we consider are the following:
\begin{itemize}
\item Model $1$: $X\sim M(\pi,1)$, $L  \text{ and } U \text{ are independent}$, $L \sim  M\left(\frac{2\pi}{3},1\right)$ and $U \sim  M\left(\frac{4\pi}{3}, 1\right)$.
\item Model $2$: $X\sim M(\pi,1)$, $L \text{ and } U \text{ are independent}$, $L \sim  M\left(\frac{4\pi}{3},1\right)$ and $U \sim  M\left(\frac{2\pi}{3}, 1\right)$. This is Model $1$ where we exchanged the role of $L$ and $U$.
\item Model $3$: $X\sim \frac{6}{10}M\left(\frac{\pi}{3},3\right) +  \frac{4}{10}M\left(\frac{15\pi}{9},3\right)$, $L  \text{ and } U \text{ are independent}$, \\$L \sim  M\left(\frac{2\pi}{3},3\right)$ and $U \sim  M\left(\frac{4\pi}{3}, 3\right)$.
\item Model $4$: $X\sim M(\pi,1)$, $L$ and $U$ are independent, $L \sim \mathcal{U}\left(\left[0 -\frac{\pi}{12},\pi + \frac{\pi}{12} \right]\right)$ and $U \sim \mathcal{U}\left(\left[ \pi - \frac{\pi}{12}, 0 +\frac{\pi}{12} \right]\right)$,
where $\mathcal{U}([a,b])$ is the uniform circular distribution on the circular interval $[a,b]$.
\end{itemize}
Table~\ref{table:1} summarizes the simulation results, reporting the estimated MISE, the average proportion of censored observations, and the mean lengths of the censoring arcs for each setting. Figure~\ref{fig:Simu_de_base} from (A) to (D) displays, for each model, the density and the estimator.  In Table~\ref{table:compare} we compare our estimator $\tilde{f}_{\hat{m}^*}$ with the one proposed in \citelink{Conanec-2025}{Conanec (2025)} using the same data sample for each model.\\
%
\begin{table}[h]
\begin{center}
\begin{tabular}{|c c c c|c| c|}
\hline {} & {$n=50$} & {$n=100$} & {$n=500$} & $\%$ cens. &  length cens. arc\\
\hline
&\multicolumn{3}{c|}{MISE estimation}& &\\
{Model 1} & {$0.038$} & {$0.021$} & {$0.006$} &44.2\% & 3.46\\
{Model 2} & {$0.036$} & {$0.025$} & {$0.006$} & 55.5\%& 2.80\\
{Model 3} & {$0.253$} & {$0.128$} & {$0.055$} & 85.9\% & 4.10\\
{Model 4} & {$0.043$} & {$0.023$} & {$0.004$} & 35.2\% & 3.15\\
\hline
\end{tabular}
\end{center}
\caption{MISE estimation, mean percentage of censored data and mean length of the censoring arc for $N=100$ replications of simulated data of different sample size for each model.}
\label{table:1}
\end{table}
\begin{figure}[!h]
\begin{subfigure}{.45\textwidth}
  \centering
   \includegraphics[width=6.5cm]{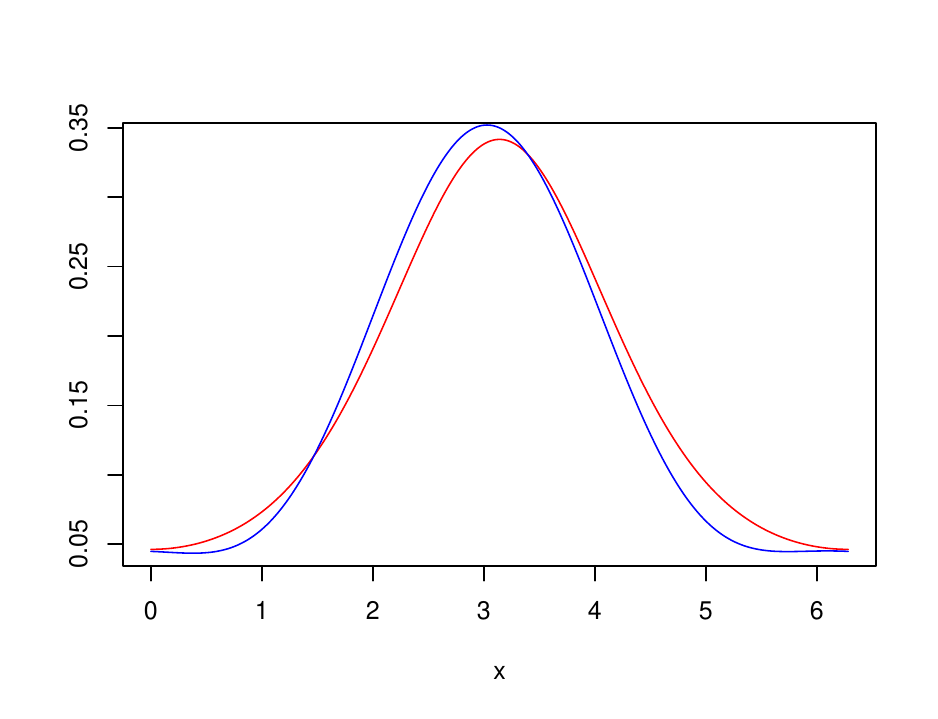}  
  \caption{Model $1$.}
  \label{fig:Sub_1_monitoring}
\end{subfigure}
\begin{subfigure}{.45\textwidth}
  \centering
  \includegraphics[width=6.5cm]{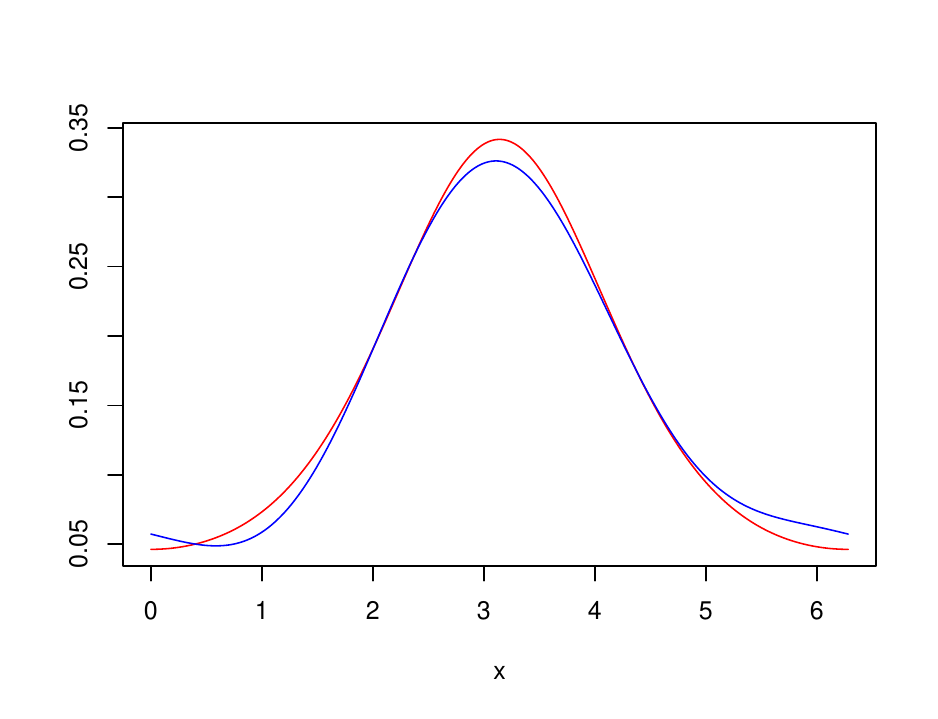}  
  \caption{Model $2$.}
  \label{fig:Sub_1_monitoring}
\end{subfigure}
\begin{subfigure}{.45\textwidth}
  \centering
  \includegraphics[width=6.5cm]{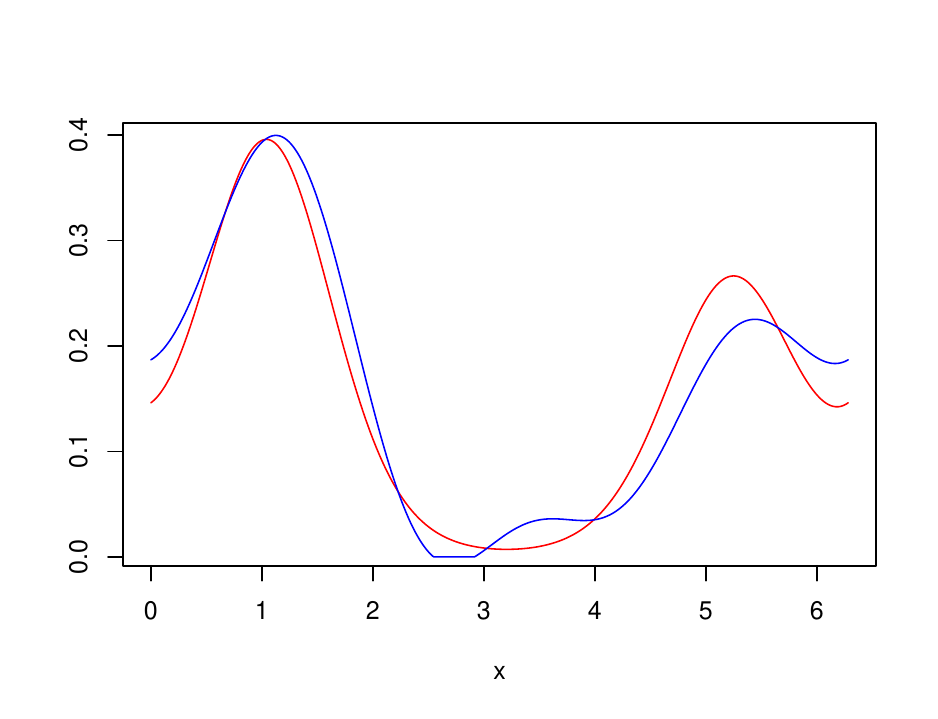}    
  \caption{Model $3$.}
  \label{fig:Sub_1_monitoring}
\end{subfigure}
\begin{subfigure}{.45\textwidth}
  \centering
  \includegraphics[width=6.5cm]{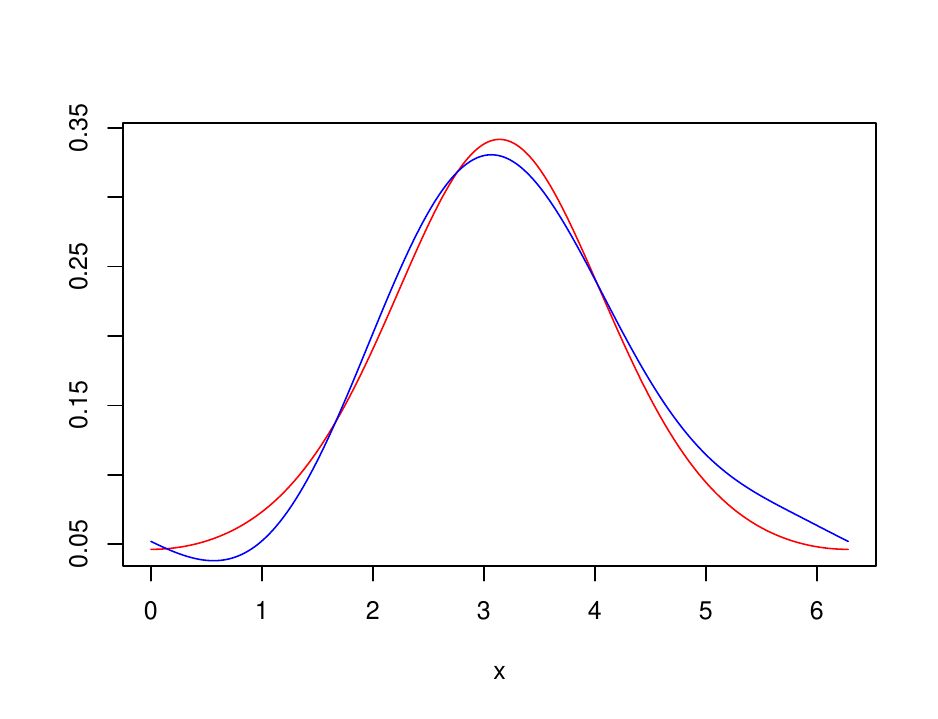}    
  \caption{Model $4$.}
  \label{fig:Sub_1_monitoring}
\end{subfigure}
\begin{subfigure}{.45\textwidth}
  \centering
  \includegraphics[width=6.5cm]{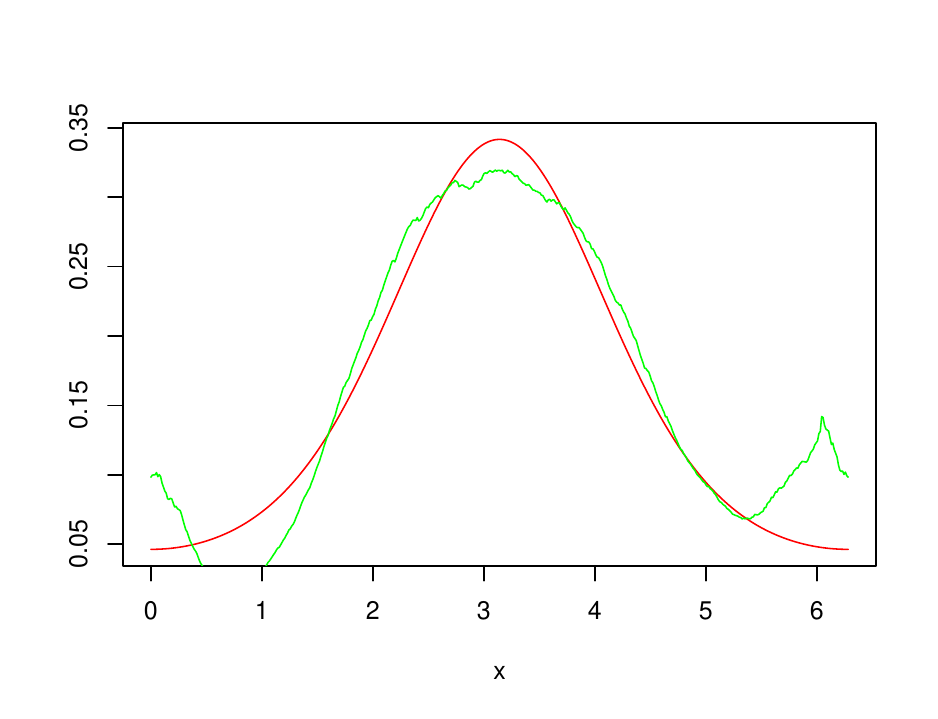}    
  \caption{Model $4$.}
  \label{fig:Sub_1_monitoring}
\end{subfigure}
\caption{Plots of the true density in red, the estimator $\tilde{f}_{\hat{m}^*}$ in blue and the estimator of \textcolor{blue}{Conanec (2025)} in green for a sample of size $n=500$ simulated data.} 
\label{fig:Simu_de_base}
\end{figure}
\begin{table}[h]
\begin{center}
\begin{tabular}{|c c c c c|}
\hline {Sample size} & {$50$} & {$100$} & {$200$} & {$500$}  \\
\hline\multicolumn{5}{|c|}{Model $1$}\\
{$\hat{f}_{\hat{m}}$} & {$0.582$} & {$0.033$} & {$0.018$} & {$0.009$} \\
{$\tilde{f}_{\hat{m}^*}$} & \boldmath{$0.438$} &  \boldmath{$0.020$} &  \boldmath{$0.013$} &  \boldmath{$0.005$} \\
\hline \multicolumn{5}{|c|}{Model $2$} \\
{$\hat{f}_{\hat{m}}$} & \boldmath{$0.040$} & \boldmath{$0.016$} & {$0.013$} & {$0.005$} \\
{$\tilde{f}_{\hat{m}^*}$} & {$0.042$} & {$0.021$} &  \boldmath{$0.012$} &  {$0.005$} \\
\hline\multicolumn{5}{|c|}{Model $3$}\\
{$\hat{f}_{\hat{m}}$} & {$0.682$} & {$0.547$} & {$0.228$} & {$0.079$} \\
{$\tilde{f}_{\hat{m}^*}$} &  \boldmath{$0.316$} &  \boldmath{$0.121$} &  \boldmath{$0.066$} &  \boldmath{$0.039$} \\
\hline\multicolumn{5}{|c|}{Model $4$}\\
{$\hat{f}_{\hat{m}}$} & {$0.066$} & {$0.038$} & {$0.018$} & {$0.007$} \\
{$\tilde{f}_{\hat{m}^*}$} &  \boldmath{$0.043$} &  \boldmath{$0.020$} &  \boldmath{$0.011$} &  \boldmath{$0.004$} \\
\hline
\end{tabular}
\end{center}
\caption{MISE estimation of our estimator $\tilde{f}_{\hat{m}^*}$ and the one from  \textcolor{blue}{Conanec (2025)}  $\hat{f}_{\hat{m}}$ using the same $N=100$ replications of simulated data.}
\label{table:compare}
\end{table}

Analyzing Table~\ref{table:1}, we observe that for the four models, the MISE decreases towards $0$ as $n$ increases, which is consistent with the theoretical results. The third model exhibits a higher MISE due to the strong level of censoring. 
The plots in Figure~\ref{fig:Simu_de_base} illustrate that the estimator performs well across the four models, even under high levels of censoring and for more challenging distributions, such as the bimodal case. Moreover, compared with \citelink{Conanec-2025}{Conanec (2025)}, our approach yields a smoother estimate of $f$, since it does not involve division by the empirical distribution function: see reconstructions (D) (our estimator) and (E) (his estimator) for  Model 4.  As for Table~\ref{table:compare}, it shows that except for Model 2 our estimator is always better than the one from \citelink{Conanec-2025}{Conanec (2025)}. \\
Finally, in Table~\ref{Tab:Comparaison1} we compare our estimator to the one introduced by  \citelink{Jammalamadaka-2009}{Jammalamadaka \& Mangalam (2009)}. 
We simulate $L$ from a uniform circular distribution, and $U$ is defined as $L - \alpha$, so that the censoring arc always has length $\alpha$. We then generate $200$ samples of size $100$, each drawn from a $M(2,1)$ distribution. 
\begin{table}[!h]
\begin{center}
\begin{tabular}{|c c c c c c c|}
\hline \multirow[t]{2}{*}{ Length of censoring arc $\alpha$ } & \multicolumn{3}{l}{ Estimation of $\mu$} & \multicolumn{3}{l|}{ Estimation of $k$} \\
& $\tilde{\mu}$ & $\hat{\mu}$ & $\mathring{\mu}$ & $\tilde{k}$  & $\hat{k}$ & $\mathring{k}$ \\
\hline 1 & {$1.985$} & {$2.005$} & {$2.005$} & {$1.047$} & {$1.041$} & {$1.044$} \\
 3 & {$2.006$} & {$1.990$}  & {$1.999 $}& {$1.037$} & {$1.021$} & {$0.986$} \\
\hline
\end{tabular}
\end{center}
\caption{Estimation of $(\mu,k)$ when $\mu=2$ and $k=1$, $X \sim M(\mu,k), L$ follows the uniform distribution and $U= L-\alpha$ with $\alpha \in \{1,3\}$.  Estimates $(\tilde{\mu},\tilde{k})$ are from  \textcolor{blue}{Jammalamadaka \& Mangalam (2009)}, $(\hat{\mu},\hat{k})$ are from  \textcolor{blue}{Conanec (2025)} and $(\mathring{\mu},\mathring{k})$ are our estimates.}
\label{Tab:Comparaison1}
\end{table}
%
%

An examination of Table~\ref{Tab:Comparaison1} shows that our estimator performs quite satisfactorily, even though, unlike \citelink{Jammalamadaka-2009}{Jammalamadaka \& Mangalam (2009)}, we do not impose a Von Mises assumption on the underlying distribution.


\section{Proofs}\setcounter{equation}{0}
\label{sec:proofs}

\subsection{Proofs notations}
~\\
~\\
\begin{tabular}{ l l }
$B^2_m(x, y)$ & Ball for the norm $\|\cdot\|_2$ in $S_m$ of center $x$ and radius $y$ \\
$<\cdot, \cdot>_2$ & Scalar product associated to the norm $\|\cdot\|_2$\\
$\|\cdot\|_{2, \sigma}$ & $\mathbb{L}^2$-norm on $\cercle$ weighted by the function $\sigma$\\
$B^{2, \sigma}_m (x, y)$ & Ball for the norm $\|\cdot\|_{2, \sigma}$ in $S_m$ of center $x$ and radius $y$ \\
$B^{2, \sigma}_{m,m'} (x, y)$& Ball for the norm $\|\cdot\|_{2, \sigma}$ in $S_m +S_{m'}$ of center $x$ and radius $y$\\
$<\cdot, \cdot>_{2, \sigma}$ & Scalar product associated to the norm $\|\cdot\|_{2, \sigma}$\\
$\|\cdot\|_{n, \sigma}$ & $\mathbb{L}^2$-{semi}norm on $\cercle$ weighted by the estimator $\hat{\sigma}$\\
$<\cdot, \cdot>_{n, \sigma}$ & {Semi-}scalar product associated to the norm $\|\cdot\|_{n, \sigma}$\\
$B_m^{n, \sigma} (x, y)$ & Ball for the {semi}norm $\|\cdot\|_{n,\sigma}$ in $S_m$ of center $x$ and radius $y$ \\
\end{tabular}

\subsection{Proof of Theorem~\ref{BorneInf}}\label{LowerBound}
 
The proof of the lower bound is based on Theorem 2.6 of \citelink{Tsybakov}{Tsybakov (2009)}. To this end, set an integer $M>1$ that will be defined later. We need first to define $M+1$ hypotheses. 
Recalling that $(\varphi_l)_l$ is the trigonometric basis of $\mathbb{L}^2(\cercle)$, our hypotheses are defined as follows
\begin{equation}\nonumber
		f_j := f_{\theta^{(j)}}: x\mapsto  \frac{1}{2\pi} +\gamma \sum_{l=1}^{\mathcal{K}} \omega^{(j)}_l \varphi_l (x), \forall j\in \{ 0, \dots , M \} ,
\end{equation}
where $\omega^{(0)} = (0,\cdots, 0)$, i.e $f_0: x\mapsto \frac{1}{2\pi}$, $\omega^{(j)} \in \{0, 1\}^{\mathcal{K}}$ and the constants $\gamma$ and $\mathcal{K}$ being defined later.
This means that the $(\gamma\omega^{(j)}_l)_l$ are the Fourier coefficients of the function $f_j$ where we forced the first coefficient to be equal to $\frac{1}{\sqrt{2 \pi}}$.\\
We write $\Proba_{j, n}$ the distribution of a sample of size $n$ of triplets $(X_i', L_i, U_i)$, when the $f_j$ is the density function of the $X_i$'s. 
To apply Theorem 2.6 of \citelink{Tsybakov}{Tsybakov (2009)}, we need to verify the four next assertions
\begin{equation}\nonumber
\left\{\begin{array}{ll}
        (1) \hspace{6pt} f_j \in W(\beta,R) ,\\
	(2) \hspace{6pt} f_j(x) \geq 0 , \forall x\in\cercle \text{ and } \Inte f_j(x)\dd x =1,\\
	(3) \hspace{6pt}\|f_j - f_i\|_2 \geq 2A n^{\frac{-\beta}{2\beta+1}} , \forall i\ne j,\\
        (4) \hspace{6pt} \frac{1}{M}\sum_{j=1}^M\chi^2(\Proba_{j,n}, \Proba_{0,n}) \leq \frac{M}{4} < \infty ,
    \end{array}
\right.
\end{equation}
where $W(\beta,R)$ is the Sobolev class of functions recalled in \eqref{Sobolev_Def},  $A$ is a real number we will define later. The $\chi^2$ divergence is defined for $P \ll Q$ as the following $$\chi^2(P, Q) = \int \left(\frac{\dd P}{\dd Q}\right)^2 \dd Q -1 = \int \left( \frac{\dd P }{\dd Q} - 1\right)^2 \dd Q.$$
 We begin by checking  the first condition.

1/ $f_j \in W(\beta , R)$.\\
We need to show that $ \sum_{l=0}^{\mathcal{K}} \alpha_l^2 |<f_j, \varphi_l >_2|^2 < \frac{R^2}{\pi^{2\beta}}$, with

\begin{equation}\nonumber
\alpha_l=\left\{\begin{array}{ll}
        l^{\beta}, \hspace{25pt}\text{for even }l,\\
        (l+1)^{\beta}, \hspace{0pt}\text{for odd }l.
    \end{array}
\right.
\end{equation}
We have 
\begin{align*}
 \sum_{l=0}^{\mathcal{K}} \alpha_l^2  |<f_j, \varphi_l >_2|^2 = \sum_{l=1}^\mathcal{K} \alpha_l^2 \gamma^2 \omega_l^2  \leq \gamma^2 \sum_{l=1}^{\mathcal{K}} (l+1)^{2\beta} &\leq \gamma^2 \mathcal{K} (\mathcal{K}+1)^{2\beta} \\
 &\leq 2^{2\beta}\gamma^2 \mathcal{K}^{2\beta +1}.
\end{align*}
We want $2^{2\beta}\gamma^2 \mathcal{K}^{2\beta + 1} < \frac{R^2}{\pi^{2\beta}}$. We will then set
\[ \gamma^2 = c_1 \mathcal{K}^{-2\beta - 1},\]
where $c_1 < \frac{R^2}{(2\pi)^{2\beta}}$.\\

2/ $f_j$'s integral value and sign.\\
Since $f_j$ is a finite linear combination of functions that integrate to $0$ and the constant  $\frac{1}{2\pi}$ integrates to $1$, $f_j$'s integral is equal to $1$.\\
To verify that $f_j$ is nonnegative, we need to compute the following upper bound

\begin{align*}
| f_j (x) - f_0 (x)|^2 &=\left | \sum_{l=1}^{\mathcal{K}} \gamma \omega^{(j)}_l \varphi_l(x)\right|^2 \leq \left(  \sum_{l=1}^{\mathcal{K}} \gamma^2 {\omega^{(j)}_l}^2\right) \left(  \sum_{l=1}^{\mathcal{K}} \varphi_l(x)^2\right)\\
& \leq \mathcal{K}\gamma^2  \sum_{l=1}^{\mathcal{K}} {\omega^{(j)}_l}^2 \leq \mathcal{K}^2 \gamma^2.
\end{align*}
Thus we have
\[ f_j \geq f_0 - \|f_j - f_0\|_{\infty} \geq \frac{1}{2\pi} - \mathcal{K}\gamma = \frac{1}{2\pi} - \sqrt{c_1}\mathcal{K}^{-\beta + \frac{1}{2}}.\]
For $\beta \geq \frac{1}{2}$ ,
\[ \frac{1}{2\pi} - \sqrt{c_1}\mathcal{K}^{-\beta + \frac{1}{2}} \geq \frac{1}{2\pi} - \sqrt{c_1}.\]
Hence, if one wants $f_j$  to be nonnegative it is sufficient to have 
\[ c_1 \leq \frac{1}{4\pi^2}.\]

3/ Separation rate.\\
For $j,i \in \{0,\dots,M\}$, we have
\[ \|f_j - f_i\|_2^2 = \|\sum_{l=1}^\mathcal{K} \gamma({\omega^{(j)}_l} - {\omega^{(i)}_l})\varphi_l\|_2^2 = \gamma^2 \sum_{l=1}^\mathcal{K} ({\omega^{(j)}_l} - {\omega^{(i)}_l})^2.\]
We now use Lemma $2.9$ of \citelink{Tsybakov}{Tsybakov (2009)} which gives the Varshamov-Gilbert bound. We recall it here
\begin{lemma}
Let $\mathcal{K}\geq 8$. There exists a subset $\left(\omega^1, \cdots , \omega^M \right) \subset \{0, 1\}^\mathcal{K}$ such that $\omega^0 = (0, \dots, 0)$,
\[ \sum_{l=1}^\mathcal{K} (\omega^i_l - \omega^j_l)^2 \geq \frac{\mathcal{K}}{8}, \hspace{5pt} \forall 1\leq i < j \leq M, \]
and $M \geq 2^{\frac{\mathcal{K}}{8}}$.
\end{lemma}
Suppose $\mathcal{K}\geq 8$. We have 
\[ \|f_j - f_i\|_2^2 \geq c_1 \mathcal{K}^{-2\beta - 1} \frac{\mathcal{K}}{8} = \frac{c_1}{8}\mathcal{K}^{-2\beta}. \]
We want each hypothesis to be sufficiently distant from one another, i.e
\[ \frac{c_1}{8}\mathcal{K}^{-2\beta} \geq 4A^2n^{\frac{-2\beta}{2\beta + 1}}.\]
So if we set $\mathcal{K}=\left\lfloor n^{\frac{1}{2\beta+1}}\right\rfloor$ and $A\leq\left(\frac{c_1}{32}\right)^{\frac{1}{2}}$
we then have
\[ \|f_j - f_i\|_2\geq 2 An^{\frac{-\beta}{2\beta +1}}\]

4/ Chi-squared divergence.\\
We want to compute the distribution of the triplet $(X', L, U)$. We recall that $X'=\Delta X - \pi(1-\Delta)$ with $\Delta= \mathds{1}_{\{X\in [L, U]\}}$ giving to $X'$ the value $X$ if it is not censored and the value $-\pi$ otherwise. The distribution of $\Delta$ is a Bernoulli law, with parameter
\[ \Esp(\mathds{1}_{\{X\in [L, U ]\}}) = \Inte\Inte\Inte \mathds{1}_{\{x\in [l, u ]\}} f(x) \Proba_{(L, U )}(\dd l ,\dd u) \dd x = \Inte f(x) \sigma(x) \dd x.\]
%
We can show that the distribution of the observations triplet $(X', L, U)$ is
\begin{align*} \Proba_{(X', L, U)}(\dd x, \dd l , \dd u) &= f(x)\mathds{1}_{\{x\in [l, u]\}}\Proba_{(L, U)}(\dd l ,\dd u) \dd x \\
&+ \left( \Inte f(t)\mathds{1}_{\{t\notin [l, u]\}} \dd t\right)\Proba_{(L, U)}(\dd l , \dd u)\delta_{-\pi}(\dd x).
\end{align*}
To justify this result, set $h$ a positive and measurable function. We have
\begin{align*}
\Esp(h(X', L, U))&= \Esp(h(X, L, U) \mathds{1}_{\{X\in [L, U]\}}) + \Esp(h(-\pi, L, U) \mathds{1}_{\{X \notin [L, U]\}})\\
&= \Inte\Inte\Inte h(x, l, u) \mathds{1}_{\{x\in [l, u]\}} f(x) \Proba_{(L, U)}(\dd l , \dd u) \dd x \\
&\quad+\Inte\Inte\Inte h(-\pi, l, u) \mathds{1}_{\{x\notin [l, u]\}} f(x) \Proba_{(L, U)}(\dd l , \dd u) \dd x  \\
&= \Inte\Inte\Inte h(x, l, u) \mathds{1}_{\{x\in [l, u]\}} f(x)\Proba_{(L, U)}(\dd l , \dd u)\dd x \\
&\quad+\Inte\Inte h(-\pi, l, u)\left(\Inte \mathds{1}_{\{t\notin [l, u]\}} f(t) \dd t\right) \Proba_{(L, U)}(\dd l , \dd u) \\
&= \Inte\Inte\Inte h(x, l, u) \mathds{1}_{\{x\in [l, u]\}} f(x)\Proba_{(L, U)}(\dd l , \dd u) \dd x \\
&\quad+\Inte\Inte \Inte h(x, l, u) \left(\Inte \mathds{1}_{\{t\notin [l, u]\}} f(t) \dd t\right)\Proba_{(L, U)}(\dd l , \dd u) \delta_{-\pi}(\dd x) .\\
\end{align*}
We now know the distribution of our observations triplet.  Our observations being $i.i.d$, the measure of our sample of size $n$ is the product of the measure of one single observation. So we only need to compute $\chi^2(\Proba_{j, 1}, \Proba_{0, 1})$. To lighten the notations we write $\Proba_j$ the distribution of $(X_1,L_1,U_1)$ when $f_j$ is the density of $X_1$. We first need to compute the derivative of Radon-Nikodym of $\Proba_j$ with respect to  $\Proba_0$. So we are looking for a function $k_{j, 0}$ such that for every borelian $\mathbf{A}$ of $\left(\mathbb{S}^1\right)^3$,
\[ \Proba_j (\mathcal{\mathbf{A}}) = \Inte\Inte\Inte k_{j, 0}(x, l, u) \mathds{1}_\mathbf{A} (x, l, u) \Proba_{0, 1} (\dd x , \dd l , \dd u). \]
We write $I_j (l, u) = \Inte f_j(t)\mathds{1}_{\{t\notin [l, u]\}}\dd t$ to make computation more clear.  
For $\mathbf{A}=A_1\times A_2 \times A_3$ a borel set of $\left(\mathbb{S}^1\right)^3$
\begin{align*}
\Proba_j(\mathbf{A}) &= \Inte\Inte\Inte \mathds{1}_\mathbf{A} (x, l, u) \Proba_j (\dd x , \dd l , \dd u) \\
&= \Inte \Inte \Inte  \mathds{1}_\mathbf{A} (x, l, u) f_j(x)  \mathds{1}_{\{x\in[l, u]\}}\Proba_{(L, U)}(\dd l , \dd u) \dd x  \\
&+ \Inte \Inte \mathds{1}_{A_2 \times A_3} (l, u) \delta_{-\pi}(A_1) I_i(l, u) g(l, u) \dd l \dd u\\
&= \Inte \Inte \Inte \mathds{1}_\mathbf{A} (x, l, u) \left( \frac{f_j(x)}{f_0(x)}\right) f_0(x)  \mathds{1}_{\{x\in[l, u]\}}\Proba_{(L, U)}(\dd l , \dd u) \dd x  \\
&+  \Inte \Inte \mathds{1}_{A_2 \times A_3} (l, u) \delta_{-\pi}(A_1) \left(\frac{I_1(l, u)}{I_0(l, u)}\right) I_0(l, u)\Proba_{(L, U)}(\dd l , \dd u)\\
&= \Inte\Inte\Inte \mathds{1}_\mathbf{A}(x, l, u) k_{j, 0}(x, l, u) \Proba_0(\dd x, \dd l, \dd u),
\end{align*}
if we define $k_{j, 0}(x, l, u) = \frac{f_j(x)}{f_0(x)}\mathds{1}_{\{x\ne -\pi\}} + \frac{I_j(l, u)}{I_0(l, u)}\mathds{1}_{\{x=-\pi\}}$.  

Now we can compute the $\chi^2$ divergence between $\Proba_j$ and $\Proba_0$
\begin{align*}
\chi^2(\Proba_j,\Proba_0) &= \int \left(\frac{\dd\Proba_j}{\dd\Proba_0} -1\right)^2 \dd\Proba_0
=\int \left(\frac{\dd\Proba_j - \dd\Proba_0}{\dd\Proba_0} \right)^2 \dd\Proba_0 \\
&= \Inte\Inte\Inte \left(\frac{f_j(x)-f_0(x)}{f_0(x)}\right)^2 f_0(x) \mathds{1}_{\{x\ne - \pi\}} \mathds{1}_{\{x\in [l, u]\}} \dd x \Proba_{(L, U)}(\dd l , \dd u) \\
&\quad + \Inte\Inte\Inte  \left(\frac{I_j(l, u) - I_0(l, u)}{I_0(l, u)}\right)^2 I_0(l, u) \Proba_{(L, U)}(\dd l , \dd u) \\
&= \Inte \frac{(f_i(x)-f_0(x))^2}{f_0(x)}\sigma(x) \dd x\\
&\quad+ \Inte \Inte \left(\Inte (f_j(t) - f_0(t))\mathds{1}_{\{t\notin [l, u]\}}) \dd t\right)^2 \frac{1}{I_0(l, u)}\Proba_{(L, U)}(\dd l , \dd u) \\
&\leq\Inte 2\pi \left(\sum_{l=1}^{\mathcal{K}}\gamma \omega_l^{(j)}\varphi_l (x)\right)^2 \dd x\\
&\quad + \Inte \Inte \left(\left[\Inte \left(\sum_{l=1}^{\mathcal{K}}\gamma \omega_l^{(j)}\varphi_l (t)\right)^2 \dd t\right] \underbrace{\left[\Inte\mathds{1}_{\{t\notin [l, u]\}}^2 \dd t\right]}_{= 2\pi I_0(l, u)}\right)\frac{1}{I_0(l, u)}\Proba_{(L, U)}(\dd l , \dd u) \\
&\leq 2\pi \gamma^2 \sum_{l=1}^{\mathcal{K}} (\omega_l^{(j)})^2 + \gamma^2\left( \sum_{l=1}^{\mathcal{K}} (\omega_l^{(i)})^2 \right)\Inte\Inte 2\pi\Proba_{(L, U)}(\dd l , \dd u)\\
&\leq 4\pi \gamma^2 \mathcal{K}
\leq 4\pi c_1 \mathcal{K}^{-2\beta -1}\mathcal{K}.
\end{align*}
To conclude the computations, we use a property of the $\chi^2$ divergence to deduce the distance between $\Proba_{j, n}$ and $\Proba_{0, n}$. Indeed we have
$\Proba_{j, n} = \Proba_j^{\otimes n}$ 
and $\Proba_{0, n} = \Proba_0^{\otimes n}$,
then using $\mathcal{K}\leq n^{\frac{1}{2\beta +1}} \leq \mathcal{K}+1$ and $\mathcal{K} \leq \frac{8\log(M)}{\log(2)}$, we obtain 
\begin{align*}
\chi^2(\Proba_{j, n}, \Proba_{0, n}) &= \prod_{i=1}^n (1+\chi^2(\Proba_j , \Proba_0))-1 = (1 + \chi^2(\Proba_j, \Proba_0))^n  -1\\
&= \exp\left(n \log\left[1+\chi^2(\Proba_j, \Proba_0)\right] \right) -1 \leq \exp(n\chi^2(\Proba_j, \Proba_0)) -1\\
&\leq \exp(4\pi c_1 \mathcal{K} \mathcal{K}^{-2\beta -1}n) -1 \leq \exp\left( \frac{2^{2\beta +6} \pi}{\log (2)}c_1 \log(M)\right) -1 .
\end{align*}
We want $\chi^2(\Proba_j, \Proba_0) \leq \frac{M}{4}$. If we take $c_1 \leq \frac{\log(2)}{2^{2\beta + 7}\pi}$ to have $\chi^2(\Proba_i, \Proba_0) \leq \exp\left( \frac{1}{2}\log(M)\right) - 1 = \sqrt{M} - 1 \leq \frac{M}{4}$. \\
So we take $c_1 = \min\left( \frac{R ^2}{(2\pi)^{2\beta}}, \frac{1}{4\pi^2}, \frac{\log(2)}{2^{2\beta +7}\pi} \right)$.
With this choice for $c_1$, the four assertions of Theorem 2.6 of \citelink{Tsybakov}{Tsybakov (2009)} are fulfilled and we deduce for any estimator $\hat{f}$,
\[ \max_{f\in \{f_0, \cdots, f_M\} } \Proba_f \left(\|\hat{f}-f\|_2 \geq An^{\frac{-\beta}{2\beta+1}}\right) \geq \frac{1}{2}\left(\frac{3}{4} - \frac{1}{M}\right), \]
which  implies the result of Theorem~\ref{BorneInf}.

\subsection{Lemmas on $\sigma$ and $\hat{\sigma}$.}

\subsubsection{A concentration inequality lemma.}\label{Preuve du lemme}

\begin{lemma}\label{Lemme}
Recall that the functions $\sigma$ and $\hat{\sigma}$ have been defined in \eqref{eq:def_sigma} and \eqref{eq:def_hat_sigma}. We have for all $y>0$,
\[ \Proba\left(\|\hat{\sigma} - \sigma\|_{\infty} \geq y\right) \leq 6 e^{-2n \frac{y^2}{9}}. \]
\end{lemma}
The proof uses Hoeffding's and Dvoretzky-Kiefer-Wolfowitz's inequalities and can be found in \citelink{Conanec-2025}{Conanec (2025)}.

\subsubsection{A useful inclusion}\label{Lemme sur xisigma0}

\begin{lemma}\label{lemmexisigma0}
Suppose Assumption (A) and let $\alpha \in (0, 1]$. If we define the random set $\Xi_{\alpha\sigma_0}=\left\{\omega\in \Omega, \|\hat{\sigma} - \sigma\|_{\infty} \leq \frac{\alpha\sigma_0}{2} \right\}$ then we have the following inclusion
\[  \Xi_{\alpha\sigma_0} \subset \left\{ \sup_{t\in \mathbb{L}^2(\cercle) , t\ne 0} \left| \frac{\|t\|^2_{n, \sigma}}{\|t\|^2_{2, \sigma}} - 1\right| \leq \frac{\alpha}{2} \right\}. \]
\end{lemma}

\begin{proof}

Let $\omega\in\Xi_{\alpha\sigma_0}$. So we know that for all $x$ in $\cercle$, we have $ -\frac{\alpha\sigma_0}{2} \leq \hat{\sigma}(x) - \sigma(x) \leq \frac{\alpha\sigma_0}{2}$.
Let $t$ in $\mathbb{L}^2(\cercle)$. Thus we have for all $x$ in $\cercle, -\frac{\alpha\sigma_0}{2}t^2(x) \leq (\hat{\sigma}(x) - \sigma(x))t^2(x) \leq \frac{\alpha\sigma_0}{2}t^2(x)$. We then use the fact $\sigma_0 \leq \sigma(x)$ for all $x$ in $\cercle$. This fact used with the last inequalities means that for all $x$ in $\cercle$,
\[ \frac{-\alpha\sigma(x)}{2}t^2(x) \leq -\frac{\alpha\sigma_0}{2}t^2(x) \leq t^2(x)\hat{\sigma}(x) - t^2(x)\sigma(x) \leq \frac{\alpha\sigma_0}{2}t^2(x) \leq \frac{\alpha\sigma(x)}{2}t^2(x).\]
Next we integrate over $\cercle$ and by the definition of the norms $\|\cdot\|_{n,\sigma}$ and $\|\cdot\|_{2,\sigma}$ we have
\[ -\frac{\alpha\|t\|_{2, \sigma}^2}{2} \leq \|t\|_{n, \sigma}^2 - \|t\|_{2, \sigma}^2 \leq \frac{\alpha\|t\|_{2, \sigma}^2}{2}.\]
Meaning that
\[ \left| \|t\|_{n, \sigma}^2 - \|t\|_{2, \sigma}^2 \right| \leq \frac{\alpha\|t\|_{2, \sigma}^2}{2}.\]
If $t\ne 0$ we finally obtain
\[\left| \frac{\|t\|^2_{n, \sigma}}{\|t\|^2_{2, \sigma}} -1\right| \leq \frac{\alpha}{2}.\]
\end{proof}

\subsection{Lemmas on $G_m$ and $\hat{G}_m$.}

\subsubsection{Proof of Lemma~\ref{def_coefs}}\label{preuve_def_coefs}
\begin{proof}
We recall the definition of the contrast
\[ \zeta(t) := \frac{1}{n}\sum_{i=1}^n \left(\Inte t^2(x)\mathds{1}_{x\in[L_i, U_i]} \dd x - 2\Delta_i t(X_i') \right).\]
Set $t\in S_m$, so $t=\sum_{\indices} b_{\lambda}\philambda = {}^t B \overrightarrow{\varphi_m}$ where $B := (b_{\lambda})_{\indices} = (<t, \philambda>_2)_{\indices}$ and $\overrightarrow{\varphi_m} = (\philambda)_{\indices}$. We note that the first term of the contrast is exactly equal to $\|t\|_{n, \sigma}^2 = {}^t B  \hat{G}_m B$. For the second term we recall the definition of the vector $\hat{U}_m$ 
\begin{equation}\label{eq:Def_U_chapeau}
\hat{U}_m = \left(\frac{1}{n}\sum_{i=1}^n \Delta_i \philambda(X_i)\right)_{\indices},
\end{equation}
which entails that the second term of the contrast is equal to $-2 \hspace{2pt} {}^t B \hat{U}_m$. This gives 
\begin{equation}\label{eq:Contraste}
 \zeta(t) =   {}^t B  \hat{G}_m B - 2\hspace{2pt} {}^t B \hat{U}_m .
 \end{equation}

We evaluate the contrast as a function of the Fourier coefficients of $t$ and identify the coefficients for which its gradient vanishes. Using the expression of $t$ in the basis $ (\philambda)_{\indices}$ we have
\begin{align*}
\zeta(t) &= \frac{1}{n}\sum_{i=1}^n \left( \sum_{\lambda,\lambda' \in \Lambda_m} b_{\lambda}b_{\lambda'}\int_{[L_i, U_i]} \varphi_{\lambda}(x) \varphi_{\lambda'}(x) \dd x - 2\sum_{\indices}\Delta_i b_{\lambda}\philambda(X_i') \right).
\end{align*}
For a $\lambda$ fixed in $\Lambda_m$, computing the partial derivative with respect to $b_{\lambda}$ gives 
\begin{align*}
\frac{\partial \zeta(t)}{\partial b_{\lambda}} &= \frac{1}{n}\sum_{i=1}^n \left( 2 b_{\lambda} \int_{[L_i, U_i]} \varphi_{\lambda}^2(x) \dd x + 2\sum_{\lambda' \ne \lambda} b_{\lambda'}\int_{[L_i, U_i]} \varphi_{\lambda}(x) \varphi_{\lambda'}(x) \dd x - 2\Delta_i\philambda (X_i')\right).\\
&= \frac{1}{n}\sum_{i=1}^n \left( 2\sum_{\lambda' } b_{\lambda'}\int_{[L_i, U_i]} \varphi_{\lambda}(x) \varphi_{\lambda'}(x) \dd x - 2\Delta_i\philambda (X_i')\right).
\end{align*}
We recall from~\eqref{eq:estimateur_argmin}
\[ \mathring{f}_m = \argmin{t\in S_m} \hspace{3pt} \zeta(t) = {}^t \hat{A}_m \overrightarrow{\varphi_m} = \sum_{\indices} \hat{a}_\lambda \philambda.\]
Since it is defined as a minimum on $\mathbb{R}^{D_m}$ we have  $\frac{\partial \zeta(\mathring{f}_m)}{\partial \hat{a}_{\lambda}} =0$ for all $\indices$.
We have the following equivalences:
\begin{align*}
\frac{\partial \zeta(\mathring{f}_m)}{\partial \hat{a}_{\lambda}} = 0 &\Leftrightarrow 
\frac{1}{n}\sum_{i=1}^n \sum_{\lambda' } \hat a_{\lambda'}\int_{[L_i, U_i]} \varphi_{\lambda}(x) \varphi_{\lambda'}(x) \dd x =\frac{1}{n}\sum_{i=1}^n \Delta_i\philambda (X_i')\\
&\Leftrightarrow \sum_{\lambda' \in \Lambda_m} (\hat{G}_m)_{\lambda \lambda'} \hat{a}_{\lambda'} = (\hat{U}_m)_{\lambda}.
\end{align*}
These computations entail that $\hat{A}_m$ satisfies the equality $\hat{G}_m\hat{A}_m = \hat{U}_m$. Since $\hat{G}_m$ is assumed to be invertible, the coefficients of $\mathring{f}_m$ are defined as 
\begin{equation*}
\hat{A}_m= \hat{G}_m^{-1}\hat{U}_m.
\end{equation*}
This is the result announced in the lemma.
\end{proof}

\subsubsection{Norm inequalities}\label{equivalence norms}
We recall below the relations between the various norms introduced in Section~\ref{SectionResultats} that will be useful in the sequel.
\begin{lemma}\label{ineq_de_normes}
 Set $t\in\mathbb{L}^2(\cercle)$. If $t\in S_m$ then we can write $t=\sum_{\indices} a_{\lambda}\philambda = {}^t A \overrightarrow{\varphi_m} $ for $A=(a_\lambda)_{\indices}$ and $\overrightarrow{\varphi_m} = (\philambda)_{\indices}$ and we obtain the last equalities of each lines
\[ \|t\|_2^2 = \Inte t^2(x) \dd x = {}^t AA,\]
\[ \|t\|_{2, \sigma}^2 = \Inte t^2(x) \sigma(x) \dd x = {}^t A G_m A,\]
\[ \|t\|_{n, \sigma}^2 = \frac{1}{n}\sum_{i=1}^n \Inte t^2(x) \mathds{1}_{[L_i, U_i]}(x) \dd x = \Inte t^2(x) \hat{\sigma}(x) \dd x = {}^t A \hat{G}_m A, \]
where the matrices $G_m$ and $\hat{G}_m$ are defined in \eqref{eq:Def_G} and \eqref{eq:Def_G_chapeau}.\\
For all $\omega$ in $\Omega$ and for all $t$ in $\mathbb{L}^2(\cercle)$ we have
\begin{equation}\label{eq:lien norme 2sigma et norme 2}
 \sigma_0\|t\|_2^2\leq \|t\|_{2, \sigma}^2 \leq \|t\|_2^2 ,
 \end{equation}
\begin{equation}\label{eq:lien norme nsigma et norme 2}
 \|t\|_{n, \sigma}^2 \leq \|t\|_2^2.
\end{equation}
where the first inequality of \eqref{eq:lien norme 2sigma et norme 2} is true only if Assumption (A) is verified.\\
Moreover if Assumption (A) is verified and $\omega \in \Xi_{\sigma_0}$,  with $ \Xi_{\sigma_0}$ defined  in Lemma~\ref{lemmexisigma0}, then we have
\begin{equation}\label{eq:lien norme 2sigma et norme nsigma}
\frac{1}{2}\|t\|_{2, \sigma}^2 \leq \|t\|_{n, \sigma}^2 \leq \frac{3}{2}\|t\|_{2, \sigma}^2 .
\end{equation}
\end{lemma}

\begin{proof}
Because we know for all $x$ in $\cercle , \sigma(x) \leq 1$, we have
\[ \|t\|^2_{2, \sigma} = \Inte t^2(x) \sigma(x) \dd x \leq \Inte t^2(x) \dd x = \|t\|_2^2,\]
and with Assumption (A) verified we have for all $x$ in $\cercle ,  \sigma_0 \leq \sigma(x) \leq 1$ then
\[ \|t\|^2_{2, \sigma} = \Inte t^2(x) \sigma(x) \dd x \geq \Inte t^2(x)\sigma_0 \dd x = \sigma_0\|t\|_2^2.\]
In an similar way, we know for all $x$ in $\cercle,  \hat{\sigma}(x) \leq 1$, we have
\[ \|t\|^2_{n, \sigma} = \Inte t^2(x) \hat{\sigma}(x) \dd x \leq \Inte t^2(x) \dd x = \|t\|_2^2.\]
As shown in the proof of Lemma~\ref{lemmexisigma0}, if $\omega$ is in $\Xi_{\sigma_0}$ we have
\[ -\frac{\|t\|_{2, \sigma}^2}{2} \leq \|t\|_{n, \sigma}^2 - \|t\|_{2, \sigma}^2 \leq \frac{\|t\|_{2, \sigma}^2}{2}, \]
and thus we have
\[\frac{1}{2} \|t\|_{2, \sigma}^2\leq \|t\|_{n, \sigma}^2 \leq \frac{3}{2}\|t\|_{2, \sigma}^2. \]
\end{proof}

\subsubsection{Proof of Lemma~\ref{inversible}}\label{inversion}
 Set $m\in \mathcal{M}_n$. Let us  deal first with the matrix $G_m$. Recall that ${G}_m$ is a square matrix of dimension $D_m \times D_m$ where, for $(\philambda)_{\indices}$ an orthonormal basis of $S_m$ for the norm $\|\cdot\|_2$,
\[ ({G}_m)_{jk} =  \Inte \varphi_j(x) \varphi_k(x) \sigma(x) \dd x = <\varphi_j, \varphi_k >_{2, \sigma}. \]
It is a Gram matrix for the vectors $(\philambda)_{\indices}$  in $S_m$ for the scalar product $<\cdot, \cdot>_{2,\sigma}$. As a Gram matrix it is semi-definite positive. We want to see if it is definite and thus invertible.
Let $b \in \mathbb{R}^{D_m}$. Suppose ${}^t b G_m b = 0$. \\
Let $t = \sum_{\indices} b_\lambda \philambda$. We have if Assumption (A) is verified:
\[ {}^t b G_m b = \|t\|_{2, \sigma}^2 \geq \sigma_0 \|t\|_2^2, \]
where we used \eqref{eq:lien norme 2sigma et norme 2} for the lower bound. This tells us that $\|t\|_2^2 = 0$, thus $t=0$ a.e. and so $b=0$.
This means that $G_m$ is invertible under Assumption (A).\\
Now let us consider  $\hat{G}_m$. We remind that  it is a square matrix of dimension $D_m \times D_m$ where, for $(\philambda)_{\indices}$ an orthonormal basis of $S_m$ for the norm $\|\cdot\|_2$,
\[ (\hat{G}_m)_{jk} = \frac{1}{n}\sum_{i=1}^n \int_{[L_i, U_i]} \varphi_j(x) \varphi_k(x) \dd x = <\varphi_j, \varphi_k >_{n, \sigma}. \]
This matrix is symmetric positive semi-definite as $\|\cdot\|_{n,\sigma}$ is a semi-norm. We want to see if it is definite and thus invertible.
Let $b \in \mathbb{R}^{D_m}$. Suppose ${}^t b \hat{G}_m b = 0$. \\
Let $t = \sum_{\indices} b_\lambda \philambda$. We have
\[ {}^t b \hat{G}_m b = \|t\|_{n, \sigma}^2 = \frac{1}{n}\sum_{i=1}^n \int_{[L_i, U_i]} t^2(x) \dd x. \]
Since this sum is equal to zero and each term is nonnegative we know that 
\[ \forall i\in \{ 1, \dots ,  n \}, \int_{[L_i, U_i]} t^2(x) \dd x =0.\]
Because $\Proba(L\ne U) = 1$, we know that a.s the interval $[L_i,U_i]$ has a positive length. Thus a.s $t$ is equal to zero on a non null set $E = \bigcup_{i=1}^n [L_i,U_i]$.
%
%
Since $t$ is a linear combination of trigonometric functions, we can write $t(x) = \sum_{\indices} b_\lambda \philambda (x) = \sum_{k=-D_m}^{D_m} c_k e^{ikx}$ where for $j>0, c_j = \frac{b_{2j-1} - ib_{2j}}{2}$ , $c_{-j}= \frac{b_{2j-1} + ib_{2j}}{2}$ and $c_0 = b_0$.\\
So we have $t(x) = e^{-imx}\sum_{k=-D_m}^{D_m} c_k e^{i(k+m)x} = e^{-imx}Q\left(e^{ix}\right)$ where $Q$ is a polynomial of degree at most $2m$. Such a polynomial cannot have more than $2m$ zeros unless it is the zero polynomial. Here $t$ admits an infinity of zeros since it is equal to zero on a non null set. Thus $t$ is the zero polynomial, i.e all $c_j$ are equal to zero and so are all the $b_\lambda$. \\
So we showed that ${}^t b \hat{G}_m b = 0$ implied $b=0$, this means that $\hat{G}_m$ is invertible almost surely.

\subsubsection{Results on the operator norm}\label{normes}
We recall that $\|M\|_{op}$ the operator norm of the matrix $M$ is the square root of the greatest eigenvalue of the matrix ${}^t MM$. Thanks to their definitions in \eqref{eq:Def_G} and \eqref{eq:Def_G_chapeau}, we know the matrices $G_m$ and $\hat{G}_m$ are squared and symmetric and semi-positive definite. This allows us to write $\|G_m\|_{op} = \sup_{x\in\mathbb{R}^{D_m}, x\ne 0} \frac{{}^t x G_m x}{{}^t xx}$.
Moreover they are positive-definite matrices as they are invertible if Assumption (A) is verified (see Lemma~\ref{inversible}). We can establish the following result.
\begin{lemma}
We have the following inequalities,
\begin{equation}\label{eq: normop G inverse}
\|G_m^{-1}\|_{op} = \sup_{x\in\mathbb{R}^{D_m}, x\ne 0} \frac{{}^t xx}{{}^t x G_m x} = \sup_{t\in S_m, t\ne 0} \frac{\|t\|_2^2}{\|t\|_{2, \sigma}^2},
\end{equation}

\begin{equation}\label{eq: normop hatG inverse}
 \|\hat{G}_m^{-1}\|_{op} = \sup_{x\in\mathbb{R}^{D_m}, x\ne 0} \frac{{}^t xx}{{}^t x \hat{G}_m x} = \sup_{t\in S_m, t\ne 0} \frac{\|t\|_2^2}{\|t\|_{n, \sigma}^2},
 \end{equation}
If Assumption (A) is verified then,
\begin{equation}\label{eq: minoration norme op hatGm}
 \|\hat{G}_m^{-1}\|_{op} \geq \sigma_0\|G_m^{-1}\|_{op},
 \end{equation}
 \begin{equation}\label{eq: majoration norme op Gm}
 \|G_m^{-1}\|_{op} \leq \frac{1}{\sigma_0}.
 \end{equation}
 Moreover on $\Xi_{\sigma_0}$, defined in Lemma~\ref{lemmexisigma0}
\begin{equation}\label{eq:majoration normeop hatG}
\|\hat{G}_m^{-1}\|_{op} \leq 2\|G_m^{-1}\|_{op} \leq \frac{2}{\sigma_0}.
\end{equation}

\end{lemma}

\begin{proof}
For \eqref{eq: normop G inverse}, we have
\begin{align*}
\|G_m^{-1}\|_{op} &= \sup_{x\in\mathbb{R}^{D_m}, x\ne 0} \frac{{}^t x G_m^{-1}x}{{}^t xx} 
= \sup_{x\in\mathbb{R}^{D_m}, x\ne 0} \frac{\|G_m^{-1/2}x\|_2^2}{\|x\|_2^2}
= \sup_{y\in\mathbb{R}^{D_m}, y\ne 0} \frac{\|y\|_2^2}{\|G_m^{1/2}y\|_2^2}\\
& = \sup_{y\in\mathbb{R}^{D_m}, y\ne 0} \frac{{}^t yy}{{}^t y G_m y} =  \sup_{t\in S_m, t\ne 0} \frac{\|t\|_2^2}{\|t\|_{2, \sigma}^2},
\end{align*}
where we write $y=G_m^{-1/2} x$. An analogous computation applies for \eqref{eq: normop hatG inverse}.\\
To prove \eqref{eq: minoration norme op hatGm}, we use  \eqref{eq: normop hatG inverse} and \eqref{eq:lien norme 2sigma et norme 2}
 \begin{align*}
\|\hat{G}_m^{-1}\|_{op} &=  \sup_{t\in S_m} \frac{\|t\|_2^2}{\|t\|_{n, \sigma}^2}
\geq  \sup_{t\in S_m, t\ne 0} \frac{\|t\|_2^2}{\|t\|_2^2} \\
&\geq \sup_{t\in S_m, t\ne 0} \frac{\|t\|_2^2}{\frac{1}{\sigma_0}\|t\|_{2, \sigma}^2} 
\geq \sigma_0 \|G_m^{-1}\|_{op}.
\end{align*}
To prove \eqref{eq: majoration norme op Gm}, we use  \eqref{eq:lien norme 2sigma et norme 2}
 \begin{align*}
\|G_m^{-1}\|_{op} &=  \sup_{t\in S_m} \frac{\|t\|_2^2}{\|t\|_{2, \sigma}^2} 
\leq  \sup_{t\in S_m, t\ne 0} \frac{\|t\|_2^2}{\sigma_0\|t\|_2^2} 
\leq \frac{1}{\sigma_0}. 
\end{align*}
To prove the upper bound on $\Xi_{\sigma_0}$ we use \eqref{eq: normop hatG inverse}, \eqref{eq:lien norme 2sigma et norme nsigma} and \eqref{eq: majoration norme op Gm}, we have
\begin{align*}
\|\hat{G}_m^{-1}\|_{op} &=  \sup_{t\in S_m} \frac{\|t\|_2^2}{\|t\|_{n,\sigma}^2} 
\leq  \sup_{t\in S_m, t\ne 0} \frac{\|t\|_2^2}{\frac{1}{2}\|t\|_{2,\sigma}^2} 
= 2 \underbrace{\sup_{t\in S_m, t\ne 0} \frac{\|t\|_2^2}{\|t\|_{2,\sigma}^2}}_{=\|G_m^{-1}\|_{op}}
\leq \frac{2}{\sigma_0}.
\end{align*}

\end{proof}

\subsubsection{An inequality between norms}\label{lemmechangementdebase}

\begin{lemma}\label{changement de base}

Let $(\philambda)_{\indices}$ be the orthonormal basis of $S_m$ for the dot product \\$<\cdot,\cdot>_{2}$ and $(\tilde{\varphi}_\lambda)_{\indices}$ be its orthonormal basis of $S_m$ for the dot product $<\cdot, \cdot>_{2, \sigma}$. Then we have
 \[\left\|\sum_{\indices}\tilde{\varphi}_\lambda^2 \right\|_\infty \leq \|G_m^{-1}\|_{op}\left\|\sum_{\indices}\philambda^2 \right\|_\infty .\]
\end{lemma}

\begin{proof}

We write $\overrightarrow{\tilde{\varphi}_m} = {}^t(\tilde{\varphi}_1 , \dots , \tilde{\varphi}_{D_m})$ and $\overrightarrow{\varphi_m} = {}^t(\varphi_1 , \dots , \varphi_{D_m})$.
Since they are orthonomal bases of the same space but for two different norms, there exists a transition matrix $A$ such that $\overrightarrow{\tilde{\varphi}_m} = A
\overrightarrow{\varphi_m}$.
By the orthonormal property of the basis $(\tilde{\varphi}_1 , \dots , \tilde{\varphi}_{D_m})$ we know that
\[ \Inte \overrightarrow{\tilde{\varphi}_m}(x) \hspace{1pt} {}^t\overrightarrow{\tilde{\varphi}_m}(x) \sigma(x) \dd x = I_{D_m}. \]
On the other hand we have
\[ \Inte \overrightarrow{\tilde{\varphi}_m}(x)\hspace{1pt} {}^t \overrightarrow{\tilde{\varphi}_m}(x) \sigma(x) \dd x = A \left( \Inte \overrightarrow{\varphi_m}(x) \hspace{1pt} {}^t \overrightarrow{\varphi_m}(x) \sigma(x) \dd x   \right) {}^t A =  A G_m  {}^tA. \]
Thus we have ${}^t A A = G_m^{-1}$. So we have
\begin{align*}
\sum_{\indices} \tilde{\varphi}_\lambda^2(x) &= {}^t\overrightarrow{\tilde{\varphi}_m}(x) \hspace{1pt} \overrightarrow{\tilde{\varphi}_m}(x) 
= {}^t\overrightarrow{\varphi_m}(x) \hspace{1pt}  {}^t A A  \hspace{1pt} \overrightarrow{\varphi_m}(x) \\&
= {}^t\overrightarrow{\varphi_m}(x) G_m^{-1} \hspace{1pt} \overrightarrow{\varphi_m}(x) 
\leq \|G_m^{-1}\|_{op} \sum_{\indices} \philambda^2(x).
\end{align*}
This final upper bound gives the inequality we wanted.

\end{proof}

\subsubsection{A random set inclusion lemma}\label{un autre lemme}

\begin{lemma}\label{lemme pour la fin}
If Assumption (A) is verified then, for $\alpha \leq 1$,

\[ \{\|G_m^{-1} - \hat{G}_m^{-1}\|_{op} > \alpha\|G_m^{-1}\|_{op}\} \subset \left\{\|\hat{\sigma} - \sigma\|_\infty > \frac{\sigma_0 \alpha}{2} \right\} = \Xi_{\alpha\sigma_0}^c.\]

\end{lemma}

\begin{proof}

Using Proposition  4 of \citelink{Comte-Genon-2020}{Comte \& Genon-Catalot (2020)}, because $G_m$ and $\hat{G}_m$ are a.s invertible, thanks to Lemma~\ref{inversible}, we know that for any $\gamma > 0$
\[\{\|G_m^{-1} - \hat{G}_m^{-1}\|_{op} > \gamma \|G_m^{-1}\|_{op} \} \subset \left\{ \| G_m^{-1/2}\hat{G}_m G_m^{-1/2} - I_m\|_{op} > \frac{\gamma \wedge 1}{2} \right\}.\]
We take $\gamma = \alpha$ thus $\gamma \wedge 1 = \alpha$. Then
 \[\left\{\| G_m^{-1/2}\hat{G}_m G_m^{-1/2} - I_m\|_{op} > \frac{\alpha}{2} \right\} = \left\{\sup_{t\in S_m, t\ne 0} \left| \frac{\|t\|^2_{n,\sigma}}{\|t\|^2_{2,\sigma}}  - 1\right| > \frac{\alpha}{2}\right\}. \]
Indeed 
\begin{align*}
\sup_{t\in S_m, t\ne 0} \left| \frac{\|t\|^2_{n,\sigma}}{\|t\|^2_{2,\sigma}}  - 1\right| &= \sup_{t\in S_m, t\ne 0} \left|  \frac{\|t\|^2_{n,\sigma}  - \|t\|^2_{2,\sigma}}{\|t\|_{2,\sigma}^2}\right|\\
&= \sup_{x\in\mathbb{R}^{D_m}, x\ne 0} \left|\frac{{}^t x (\hat{G}_m - G_m) x}{{}^t x G_m x}\right|
 \\&
 = \sup_{u\in\mathbb{R}^{D_m}, u\ne 0} \left|\frac{{}^t u G_m^{-1/2}(\hat{G}_m - G_m)G_m^{-1/2} u}{{}^t u G_m^{-1/2} G_m G_m^{-1/2} u}\right| \\
&= \sup_{u\in\mathbb{R}^{D_m}, u\ne 0} \left|\frac{{}^t u G_m^{-1/2}(\hat{G}_m - G_m)G_m^{-1/2} u}{{}^t uu}\right| \\
&= \| G_m^{-1/2}(\hat{G}_m - G_m) G_m^{-1/2}\|_{op}
 = \|G_m^{-1/2}\hat{G}_m G_m^{-1/2} - I_m\|_{op},
\end{align*}
where we took $x= G_m^{-1/2} u$. \\
Thanks to that equality we then have the inclusion
 \[\left\{\| G_m^{-1/2}\hat{G}_m G_m^{-1/2} - I_m\|_{op} > \frac{\alpha}{2} \right\} \subset \left\{\sup_{t\in \mathbb{L}^2(\cercle), t\ne 0} \left| \frac{\|t\|^2_{n, \sigma}}{\|t\|^2_{2, \sigma}}  - 1\right| > \frac{\alpha}{2}\right\}. \]
Using Lemma~\ref{lemmexisigma0} we have
\[  \left\{\sup_{t\in \mathbb{L}^2(\cercle), t\ne 0} \left| \frac{\|t\|^2_{n, \sigma}}{\|t\|^2_{2, \sigma}}  - 1\right| > \frac{\alpha}{2}\right\} \subset \Xi_{\alpha \sigma_0}^c ,\]
giving us the result we wanted with the definition of that random set.

\end{proof}

\subsection{Upper bound}

\subsubsection{Proof of Theorem~\ref{MISE du deuxieme estimateur}}\label{Preuve_MISE_Deuxieme_Estimateur}

\begin{proof}

Recall, for $m\in\mathcal{M}_n$ our estimator is $\tilde{f}_m =\mathring{f}_m \mathds{1}_{\Gamma_m}$ where $\Gamma_m=\{\|\mathring{f}_m\|_2^2 \leq k_n\}$ and $\mathring{f}_m =  \argmin{t\in S_m}\hspace{3pt}  \zeta(t) = \sum_{\indices} \hat{a}_{\lambda} \philambda$ where $(\hat{a}_\lambda)_{\indices}= \hat{G}_m^{-1}\hat{U}_m$ with this matrix and this vector defined in \eqref{eq:Def_G_chapeau} and \eqref{eq:Def_U_chapeau}.
Set $m\in \mathcal{M}_n$ and the following random set
%
%
\[\Xi_{\sigma_0}=\left\{\omega\in \Omega, \|\hat{\sigma} - \sigma\|_{\infty} \leq \frac{\sigma_0}{2} \right\}.\]
We recall thanks to Lemma~\ref{lemmexisigma0} that we have the inclusion
\[ \Xi_{\sigma_0} \subset \left\{ \sup_{t\in \mathbb{L}^2(\cercle) , t\ne 0} \left| \frac{\|t\|^2_{n, \sigma}}{\|t\|^2_{2, \sigma}} -1\right| \leq \frac{1}{2} \right\}.\]
With these random sets we split the risk into three terms
\[\Esp(\|f - \tilde{f}_m\|_2^2) = \underbrace{\Esp(\|f - \tilde{f}_m\|_2^2\mathds{1}_{\Xi_{\sigma_0}}\mathds{1}_{\Gamma_m})}_{=A_1} + \underbrace{\Esp(\|f - \tilde{f}_m\|_2^2\mathds{1}_{\Xi_{\sigma_0}}\mathds{1}_{\Gamma_m^c})}_{=A_2} + \underbrace{\Esp(\|f - \tilde{f}_m\|_2^2\mathds{1}_{\Xi_{\sigma_0}^c})}_{=A_3} .\]

For $A_3$ we know that $\|\tilde{f}_m\|_2^2 \leq k_n$ so we have thanks to Lemma~\ref{Lemme}, used with $y=\frac{\sigma_0}{2}$,
\[ A_3 \leq \Esp(2(\|f\|_2^2 + \|\tilde{f}_m\|_2^2)\mathds{1}_{\Xi_{\sigma_0}^c}) \leq 2(\|f\|_2^2 +k_n)\Proba(\Xi_{\sigma_0}^c) \leq 12(\|f\|_2^2 + k_n)e^{-2n \frac{\sigma_0^2}{36}} .\] 

For $A_2$ we use the fact that on $\Gamma_{m}^c, \tilde{f}_{m} \equiv 0 $. Using inequalities of Lemma~\ref{ineq_de_normes} and since on $\Xi_{\sigma_0}$ we know that if $\|t\|_2^2 > M$ then $\|t\|_{n, \sigma}^2 > \frac{\sigma_0 M}{2}$ for $t\in\mathbb{L}^2(\cercle)$ and $M>0$. Thus we have
\begin{align*}
A_2 &= \Esp(\|f - \tilde{f}_m\|_2^2\mathds{1}_{\Xi_{\sigma_0}}\mathds{1}_{\Gamma_m^c}) 
= \Esp(\|f\|_2^2\mathds{1}_{\Xi_{\sigma_0}}\mathds{1}_{\Gamma_m^c}) \\
&\leq \|f\|_2^2 \Proba(\Xi_{\sigma_0} \cap \Gamma_m^c)\\
&\leq \|f\|_2^2 \Proba\left(\Xi_{\sigma_0} \cap \left\{\|\mathring{f}_m\|_{n, \sigma}^2 > \frac{\sigma_0 k_n}{2}\right\}\right)\\
&\leq \|f\|_2^2 \frac{2}{\sigma_0 k_n}\Esp(\|\mathring{f}_m\|_{n, \sigma}^2 \mathds{1}_{\Xi_{\sigma_0}}),
\end{align*}
using Markov's inequality at the end.\\
Moreover thanks to the definition of $\mathring{f}_m$ and the inequalities of Section~\ref{normes} we have
\begin{align*}\|\mathring{f}_m\|_{n, \sigma}^2 &= {}^t \hat{U}_m\hat{G}_m^{-1}\hat{G}_m\hat{G}_m^{-1} \hat{U}_m ={}^t \hat{U}_m\hat{G}_m^{-1} \hat{U}_m \\
&\leq \|\hat{G}_m^{-1}\|_{op}\left(\sum_{\indices} |\hat{U}_m|^2_\lambda \right) \leq \frac{2}{\sigma_0} \Phi_0^2 D_m \leq \frac{n}{\pi\sigma_0},
\end{align*}
where we use Equation~\eqref{eq:Linear_Sieves2} to control $\|\hat {U}_m\|_2^2$.
Thus we have
\[A_2 \leq \|f\|_{\infty}^2 \frac{2n}{\pi\sigma_0^2 k_n} \Proba(\Xi_{\sigma_0}) \leq \|f\|_{\infty}^2 \frac{2n}{\pi\sigma_0^2 k_n}. \]
So if $k_n \geq n^2$ we have $A_2 \leq \frac{C}{n}$.\\
Lastly for $A_1$, on $\Xi_{\sigma_0}$, using the inequalities in Lemma~\ref{ineq_de_normes}  we have  
\begin{equation}\label{relation-A1}
\frac{\sigma_0}{2}\|\mathring{f}_m - f\|_2^2 \leq \|\mathring{f}_m - f\|_{n, \sigma}^2.
\end{equation}
Hence we need to find an upperbound for 
$ \|\mathring{f}_m - f\|_{n, \sigma}^2$. To this purpose, we have for two functions $t$ and $s$
\begin{align*}
\zeta(t) - \zeta(s) &= \frac{1}{n}\sum_{i=1}^n \left( \Inte (t^2(x) - s^2(x))\mathds{1}_{\{x\in[L_i, U_i]\}} \dd x - 2\Delta_i (t-s)(X_i') \right)\\
&=\frac{1}{n}\sum_{i=1}^n \bigg( \Inte (t - f)^2(x)\mathds{1}_{\{x\in[L_i,U_i]\}} \dd x -\Inte (s - f)^2(x)\mathds{1}_{\{x\in[L_i, U_i]\}} \dd x  \\
& \quad +2 \Inte (t(x) - s(x))f(x)\mathds{1}_{\{x\in[L_i, U_i]\}} \dd x - 2\Delta_i (t-s)(X_i') \bigg)\\
&= \|t-f\|_{n,\sigma}^2 - \|s-f\|_{n,\sigma}^2 \\
&\quad - \frac{2}{n}\sum_{i=1}^n \left( \Delta_i (t-s)(X_i') - \Inte (t(x)-s(x))f(x) \mathds{1}_{x\in [L_i, U_i]} \dd x \right).
\end{align*}
We define the empirical process $\tilde{\nu}(t) = \frac{1}{n}\sum_{i=1}^n \left( \Delta_i t(X_i') - \Inte t(x)f(x) \mathds{1}_{x\in [L_i, U_i]} \dd x \right)$
which also equals to $\tilde{\nu}(t) = \frac{1}{n}\sum_{i=1}^n \Delta_i t(X_i') - \Esp( \Delta_i t(X_i') | L_i, U_i)$.
Using that $\mathring{f}_m$ is a minimizer of $\zeta$ and $f_m$ is the projection of $f$ on the space of functions $S_m$, we have
\begin{align}
\nonumber 0 &\geq \zeta(\mathring{f}_m) - \zeta(f_m) = \|\mathring{f}_m - f\|_{n,\sigma}^2 - \|f_m - f\|_{n,\sigma}^2 - 2\tilde{\nu}(\mathring{f}_m - f_m)\\
\label{eq:Ineg4}&\Leftrightarrow \|\mathring{f}_m - f\|_{n,\sigma}^2 \leq \|f_m - f\|_{n,\sigma}^2 + 2\tilde{\nu}(\mathring{f}_m - f_m).
\end{align}
As $\tilde{\nu}$ is linear we obtain the following inequality
\begin{align}
 \nonumber2\tilde{\nu}(\mathring{f}_m - f_m) &= 2\|\mathring{f}_m - f_m\|_{n, \sigma}\tilde{\nu}\left( \frac{\mathring{f}_m - f_m}{\|\mathring{f}_m - f_m\|_{n, \sigma}} \right) \\
&\nonumber\leq 2 \|\mathring{f}_m - f_m\|_{n, \sigma} \left( \sup_{t\in B_m^{n, \sigma}(0, 1)} \tilde{\nu}(t) \right) \\
& \nonumber\leq \frac{1}{x}\|\mathring{f}_m - f_m\|_{n, \sigma}^2 + x\sup_{t\in B_m^{n, \sigma}(0, 1)} \tilde{\nu}^2(t) \\
&\label{eq:Ineg5} \leq \frac{1}{x}(1+y^{-1})\|\mathring{f}_m - f\|_{n, \sigma}^2 + \frac{1}{x}(1+y)\|f_m - f\|_{n,\sigma}^2 + x\sup_{t\in B_m^{n, \sigma}(0, 1)} \tilde{\nu}^2(t),
\end{align}
using Young and Cauchy-Schwarz inequalities and  $(x, y)\in(\mathbb{R}^*_+)^2$ and where $B^{n, \sigma}_m(0, 1)$ is the unit ball for the norm $\|\cdot\|_{n, \sigma}$ in $S_m$. Set $y=C_x = \frac{x+1}{x-1}$. So taking \eqref{eq:Ineg4} and injecting \eqref{eq:Ineg5} we obtain
\begin{align*}
\hspace{10pt}\|\mathring{f}_m - f\|_{n, \sigma}^2 &\leq \|f_m - f\|_{n, \sigma}^2 + 2\tilde{\nu}(\mathring{f}_m - f_m)\\
& \leq \|f_m - f\|_{n, \sigma}^2 + \frac{1}{x}(1+C_x^{-1})\|\mathring{f}_m - f\|_{n, \sigma}^2 \\
&\quad  + \frac{1}{x}(1+C_x)\|f_m - f\|_{n, \sigma}^2 + x\sup_{t\in B_m^{n, \sigma}(0, 1)} \tilde{\nu}^2(t) \\
&\hspace{-66pt}\Leftrightarrow \|\mathring{f}_m - f\|_{n, \sigma}^2 \leq C_x^2 \|f_m - f\|_{n, \sigma}^2 + x C_x \sup_{t\in B_m^{n, \sigma}(0, 1)} \tilde{\nu}^2(t).
\end{align*}
On $\Xi_{\sigma_0}$, thanks to the inequalities in Lemma~\ref{ineq_de_normes}, we know that we have the following relation $B_m^{n, \sigma}(0, 1) \subset B_m^{2, \sigma}\left(0, \sqrt{2}\right) \subset B_m^2 \left(0, \sqrt{\frac{2}{\sigma_0}} \right)$ where $B_m^x$ is a ball for the norm $\|\cdot\|_x$ in $S_m$. Thus on $\Xi_{\sigma_0}$ we have the inequality
\[ \|\mathring{f}_m - f\|_{n, \sigma}^2 \leq  C_x^2 \|f_m - f\|_{n, \sigma}^2 + x C_x \sup_{t\in B_m^2 \left(0, \sqrt{\frac{2}{\sigma_0}} \right)} \tilde{\nu}^2(t).\]
Furthermore we know that $\|f_m - f\|_{n, \sigma}^2 \leq \|f_m - f\|_2^2$ and equation \eqref{relation-A1} entails 
\[ \|\mathring{f}_m - f\|_2^2 \leq  \frac{2}{\sigma_0}C_x^2 \|f_m - f\|_2^2 + \frac{2}{\sigma_0}x C_x \sup_{t\in B_m^2 \left(0, \sqrt{\frac{2}{\sigma_0}} \right)} \tilde{\nu}^2(t) .\]
 Set $t\in B_m^2 \left(0, \sqrt{\frac{2}{\sigma_0}} \right)$ because we consider $t\in S_m$ we can write $t = \sum_{\indices} a_{\lambda}\philambda$.
\begin{align}
\nonumber
\tilde{\nu} (t) &= \frac{1}{n}\sum_{i=1}^n \sum_{\indices} \Delta_i a_{\lambda}\philambda(X_i') - \int_{[L_i, U_i]} a_{\lambda}\philambda(x) f(x) \dd x \\
\label{eq:nutilde}
&= \sum_{\indices} a_{\lambda}\tilde{\nu}(\philambda) 
\leq \sqrt{\sum_{\indices} a_{\lambda}^2} \sqrt{\sum_{\indices}\tilde{\nu}^2(\philambda)}.
\end{align}
This gives us
\begin{align*}\sup_{t\in B_m^2 \left(0, \sqrt{\frac{2}{\sigma_0}} \right)} \tilde{\nu}^2(t) &\leq  \frac{2}{\sigma_0}\sum_{\indices}\tilde{\nu}^2(\philambda) \\
&= \frac{2}{\sigma_0}\sum_{\indices}\left( \frac{1}{n}\sum_{i=1}^n \Delta_i\philambda(X_i') - \int_{[L_i, U_i]}\philambda(x) f(x) \dd x \right)^2.
\end{align*}
We define $T_{i, \lambda}  = \Delta_i \philambda(X_i')$ and $Z_{i, \lambda} = T_{i, \lambda} - \Esp( T_{i, \lambda} | L_i, U_i)$. It remains to find an upper bound for $\sum_{\indices} \Esp\left( \left[\frac{1}{n}\sum_{i=1}^n Z_{i, \lambda} \right]^2\right)$. Because of its definition we have $\Esp(Z_{i, \lambda})=0$.
Using the independency of the triplets $(X_i', L_i, U_i)$ we have for all $\indices$,

\begin{align*}
\Esp\left( \left[\frac{1}{n}\sum_{i=1}^n Z_{i, \lambda} \right]^2\right) 
&= \frac{\Esp(Z_{1, \lambda}^2)}{n}.
\end{align*}
Now we need to prove that $\Esp(Z_{1, \lambda}^2) \leq \Esp(T_{1, \lambda}^2)$. Indeed we
\begin{align}
\nonumber
\Esp(Z_{1, \lambda}^2 | L_1, U_1)) 
&= \Esp( (T_{1, \lambda} - \Esp( T_{1, \lambda} | L_1, U_1))^2 | L_1, U_1)) \\
\nonumber
&= \Esp(T_{1, \lambda}^2| L_1, U_1) - \Esp(T_{1, \lambda} | L_1, U_1)^2 \\
\label{eq:EspdeTetZ}
&\leq \Esp(T_{1, \lambda}^2| L_1, U_1).
\end{align}
Finally $\Esp(Z_{1, \lambda}^2) = \Esp(\Esp(Z_{1, \lambda}^2 | L_1, U_1))) \leq \Esp(\Esp(T_{1, \lambda}^2| L_1, U_1)) = \Esp(T_{1, \lambda}^2)$.
This gives us
\begin{align*}
\sum_{\indices} \Esp\left( \left[\frac{1}{n}\sum_{i=1}^n Z_{i, \lambda} \right]^2\right) &\leq \sum_{\indices} \frac{1}{n}\Esp(\Delta^2 \philambda^2(X'))\\
 &\leq\frac{1}{n}\Esp( (\sum_{\indices}\philambda^2(X')) \Delta) \\
 &\leq \frac{1}{n} \Esp(\|\sum_{\indices}\philambda^2 \|_{\infty}\Delta)\\
 &\leq \frac{\Phi_0^2 D_m \Esp(\Delta)}{n}.
\end{align*}
So we finally have the following inequality
\[A_1 \leq \frac{2C_x^2}{\sigma_0}\|f_m - f\|_2^2 + \frac{2xC_x D_m \Esp(\Delta)}{\pi\sigma_0^2 n}.\]
So we have in the end 
 \begin{align*}
\Esp(\|\tilde{f}_m - f\|_2^2) &\leq \frac{2C_x^2}{\sigma_0}\|f_m - f\|_2^2 + \frac{2xC_x D_m \Esp(\Delta_1)}{\pi \sigma_0^2 n} \\
&\quad+ \|f\|_{\infty}^2 \frac{2 n}{\pi\sigma_0^2 k_n} + 12(\|f\|_{\infty}^2 + k_n)e^{-2n \frac{\sigma_0^2}{36}}\\
&\leq C_1 (\sigma_0) \|f_m - f\|_2^2 + C_2 (\sigma_0) \frac{D_m}{n} + \frac{C_3 (\sigma_0, \|f\|_{\infty}) }{n}.
\end{align*}
Suppose now that $f$ is an element of $W(\beta,R)$. So if we write $f=\sum_{j=0}^{\infty} a_j\varphi_j$, using the definition of the Sobolev class reminded in Section~\ref{SectionBorneInf}, we have
\[ \|f_m - f\|_2^2 = \sum_{D_m +1}^{\infty} a_j^2 \leq \frac{1}{\alpha^2_{D_m +1}}\sum_{j=0}^{\infty} \alpha_j^2 a_j^2 \leq \frac{R^2}{\pi^{2\beta}\alpha^2_{D_m +1}} \leq \frac{R^2}{\pi^{2\beta}}D_m^{-2\beta}.\]
So  if we take $D_m = \lfloor n^{\frac{1}{2\beta +1}} \rfloor$ we have
\[ \|f_m - f\|_2^2 \leq \frac{R^2}{\pi^{2\beta}}\left( n^{\frac{1}{2\beta +1}}\right)^{-2\beta} \leq \frac{R^2}{\pi^{2\beta}} n^{\frac{-2\beta}{2\beta+1}},\]
and
\[ \frac{D_m}{n} \leq n^{\frac{1}{2\beta +1} - 1} = n^{\frac{-2\beta}{2\beta+1}}.\]
Thus we obtain 
\[ \Esp(\|\tilde{f}_m - f\|_2^2) = \mathcal{O}\left(n^{\frac{-2\beta}{2\beta+1}}\right),\]
which completes the proof of Theorem \ref{MISE du deuxieme estimateur}.

\end{proof}

\subsubsection{Proof of Proposition~\ref{Adaptation du deuxieme estimateur}}\label{Preuve_Adaptation_Deuxieme_Estimateur}

\begin{proof}
We recall the definition of our estimator $\tilde{f}_{\hat{m}} = \mathring{f}_{\hat{m}}\mathds{1}_{\Gamma_{\hat{m}}}$ where $\mathring{f}_{\hat{m}} = \argmin{t\in S_{\hat{m}}}\hspace{3pt} \zeta(t)$ and  $$\Gamma_{\hat{m}} = \{\|\mathring{f}_{\hat{m}}\|_2^2 \leq k_n\}.$$
We recall the definition of the set, defined in Section~\ref{Preuve_MISE_Deuxieme_Estimateur},
\[\Xi_{\sigma_0}=\left\{\omega\in \Omega, \|\hat{\sigma} - \sigma\|_{\infty} \leq \frac{\sigma_0}{2} \right\}.\]
We recall the result of Lemma~\ref{lemmexisigma0} 
\[ \Xi_{\sigma_0} \subset \left\{ \sup_{t\in \mathbb{L}^2(\cercle) , t\ne 0} \left| \frac{\|t\|^2_{n, \sigma}}{\|t\|^2_{2, \sigma}} -1\right| \leq \frac{1}{2} \right\}.\]
Thus we can write
\begin{equation}\label{eq:Separationenensemble}
\Esp(\|\tilde{f}_{\hat{m}} - f\|^2_2)=\underbrace{\Esp(\|\tilde{f}_{\hat{m}} - f\|^2_2\mathds{1}_{\Xi_{\sigma_0}\cap \Gamma_{\hat{m}}})}_{=A_1} + \underbrace{\Esp(\|\tilde{f}_{\hat{m}} - f\|^2_2\mathds{1}_{\Xi_{\sigma_0}\cap \Gamma_{\hat{m}}^c})}_{=A_2} + \underbrace{\Esp(\|\tilde{f}_{\hat{m}} - f\|^2_2\mathds{1}_{\Xi_{\sigma_0}^c})}_{=A_3} .
\end{equation}

We deal with $A_1$ first. We have on $\Xi_{\sigma_0}$ $$\frac{\sigma_0}{2}\|\mathring{f}_{\hat{m}} - f\|_2^2 \leq \|\mathring{f}_{\hat{m}} - f\|_{n,\sigma}^2,$$
thus we need to upperbound $ \|\mathring{f}_{\hat{m}} - f\|_{n,\sigma}^2$.
By definition of $\hat{m}$ and $\mathring{f}_m$ we know that for all $m\in \mathcal{M}_n$,
\[ \zeta(\mathring{f}_{\hat{m}}) + \text{pen}(\hat{m}) \leq \zeta(\mathring{f}_m) + \text{pen}(m) \leq \zeta(f_m) + \text{pen}(m).\]
We now have an analogous inequality as \eqref{eq:Ineg4}, 
\begin{align*}
\|\mathring{f}_{\hat{m}} - f\|_{n,\sigma}^2 &\leq \|f_m - f\|_{n,\sigma}^2 + 2\tilde{\nu}(\mathring{f}_{\hat{m}} - f_m) + \text{pen}(m) - \text{pen}(\hat{m})\\
&\leq \|f_m - f\|_{n,\sigma}^2 + x^{-1}\|\mathring{f}_{\hat{m}} - f_m\|^2_{2,\sigma} \\
&\quad + x \left(\sup_{t\in B_{m, \hat{m}}^{2, \sigma}(0, 1)} \tilde{\nu}^2 (t) \right) + \text{pen}(m) - \text{pen}(\hat{m}) \\
&\leq (1+2x^{-1})\|f_m - f\|_2^2 + 2x^{-1}\|\mathring{f}_{\hat{m}} - f\|^2_{2,\sigma} \\
& \quad + x\left(\sup_{t\in B_{m, \hat{m}}^{2, \sigma}(0, 1)} \tilde{\nu}^2 (t) - \text{p}(\hat{m}, m) \right) + x\text{p}(\hat{m},m) +\text{pen}(m) - \text{pen}(\hat{m}) \\
& \leq (1+2x^{-1})\|f_m - f\|_2^2 + 2x^{-1}\|\mathring{f}_{\hat{m}} - f\|_{2,\sigma}^2 + xR_m + 2\text{pen}(m),
\end{align*}
where $B_{m, m'}^{2, \sigma}(0, 1):= \{ t\in S_m + S_{m'}, \|t\|_{2, \sigma}=1 \}$ and we use a function $\text{p}(\cdot, \cdot)$ such that 
\begin{equation}\label{eq:petpen}
x\text{p}(m', m) \leq \text{pen}(m) + \text{pen}(m'),
\end{equation}
for all $m', m\in \mathcal{M}_n$ and $x>0$ and where 

\begin{equation}\label{eq:Om}
R_m = \sum_{m'\in\mathcal{M}_n} \left(\sup_{t\in B_{m, m'}^{2, \sigma}(0, 1)} \tilde{\nu}^2 (t) - \text{p}(m', m) \right).
\end{equation}
Then we use the fact that on $\Xi_{\sigma_0}$: 
$\|\mathring{f}_{\hat{m}} - f\|^2_{2,\sigma} \leq 2\|\mathring{f}_{\hat{m}}- f\|_{n, \sigma}^2$, thus we have
\begin{equation}\label{Om}
\frac{x-4}{x}\|\mathring{f}_{\hat{m}} -f\|^2_{n,\sigma}\mathds{1}_{\Xi_{\sigma_0}} \leq \frac{x+2}{x}\|f_m - f\|_2^2 + xR_m + 2\text{pen}(m).
\end{equation}
This last inequality forces us to choose $x>4$. We need to find an upper bound of $\Esp(R_m)$ which will be obtained using a Talagrand inequality.\\
We recall that $\tilde{\nu}(t) = \frac{1}{n}\sum_{i=1}^n \Delta_i t(X_i') - \Esp(\Delta_i t(X_i') | L_i, U_i) $. So for every function $t$, all of the terms of the sum have an expected value equal to zero. Thus we can apply this slight modification of Talagrand's lemma stated below.

\begin{lemma}\label{Talagrand_Conditionnel}
Let $(Z_1, Y_1), \dots, (Z_n, Y_n)$ be independent random variables  and 
\[ \nu(l) = \frac{1}{n}\sum_{i=1}^n l(Z_i) - \Esp(l(Z_i) | Y_i),\] 
for $l$ belonging to a countable class $\mathcal{L}$ of measurable functions. Then for $\epsilon > 0$,

\[ \Proba\left(\sup_{l\in\mathcal{L}} |\nu(l)| \geq H + \xi\right) \leq \exp\left(-\frac{n\xi^2}{2(v+4HM_1) + 6M_1\xi} \right),\]

for \hspace{0.4cm}$\underset{l\in\mathcal{L}}{\sup} \|l\|_{\infty} \leq M_1 , \hspace{0.4cm}\Esp\left(\underset{l\in\mathcal{L}}{\sup} |\nu(l)|\right) \leq H , \hspace{0.4cm}\underset{l\in\mathcal{L}}{\sup}\frac{1}{n} \sum_{i=1}^n\Var(l(Z_i)) \leq v $.\\
Moreover, for $\epsilon >0$, we have
\[ \Esp\left( \sup_{l\in\mathcal{L}} |\nu(l)|^2 - 2(1+2\epsilon)H^2\right)_+ \leq C\left(\frac{v}{n}e^{-K_1 \epsilon \frac{nH^2}{v}} + \frac{M_1^2}{n^2 C(\epsilon)^2}e^{-K_2C(\epsilon)\sqrt{\epsilon}\frac{nH}{M_1}} \right),\]
where $C(\epsilon) = (\sqrt{1+\epsilon}-1)\wedge 1$, $K_1= 1/6$ et $K_2 = 1/(21\sqrt{2})$.
\end{lemma}
Note that the proof of this lemma is analogous to a famous lemma from \citelink{Talagrand}{Talagrand (`96)}, using a concentration inequality in \citelink{Klein-Rio}{Klein \& Rio (2005)} (applied  with their notations to the difference $s^i(x) = l(x) - \Esp(l(Z_i) | Y_i)$).

Now, let us see how to apply this lemma to our empirical process $\tilde \nu$. Thanks to usual density arguments, instead of a countable class, we are allowed to take a unit ball in a finite functional space. Thus we consider the class of measurable functions $ \mathcal{L} =  \{l_t, t \in B^{2,\sigma}_{m, m'}(0, 1)\}$, its elements defined by  $l(z) = l_t(x,\delta) = \delta t(x)$ and  $Z_i = (\Delta_i, X_i')$ and $Y_i = (L_i, U_i)$. With those definitions $\tilde{\nu}$ is  properly defined as in the lemma's statement. Consequently we need to compute the three upper bounds $M_1, H$ and $v$. Since the linear sieves are nested the dimension of the space $S_m + S_{m'}$ verifies $D_{m \vee m'} \leq D_m + D_{m'}$.

\begin{align*}
\sup_{l\in\mathcal{L}} \|l\|_{\infty} &= \sup_{t\in B_{m, m'}^{2, \sigma}(0, 1)} \|l_t \|_{\infty}
\leq \sup_{t\in B_{m, m'}^{2, \sigma}(0, 1)} \|t \|_{\infty}\\
&\leq \sup_{t\in B_{m, m'}^{2, \sigma}(0, 1)} \Phi_0 \sqrt{D_{m \vee m'}} \|t\|_2 \\
& \leq  \frac{\Phi_0}{\sigma_0}  \sqrt{D_{m'}+D_m} =: M_1 .  \\
\end{align*}
Moreover
\begin{align*}
\sup_{l\in\mathcal{L}} \Var(l(Z_1)) &=  \sup_{t\in B_{m, m'}^{2, \sigma}(0, 1)}\Var(\Delta t(X))\\
&\leq \sup_{t\in B_{m, m'}^{2, \sigma}(0, 1)} \Esp( \Delta t^2(X)) \\
&= \sup_{t\in B_{m, m'}^{2, \sigma}(0, 1)} \Inte \Inte \Inte t^2(x) f(x) \mathds{1}_{x\in[l, u]} \Proba_{(L ,U)}(\dd l , \dd u) \dd x\\
&= \sup_{t\in B_{m, m'}^{2, \sigma}(0, 1)} \Inte t^2(x) f(x) \sigma(x) \dd x  \\
&\leq \|f\|_{\infty} \sup_{t\in B_{m, m'}^{2, \sigma}(0, 1)} \|t\|_{2, \sigma}^2 \leq \|f\|_{\infty} =: v .
\end{align*}
We use here the result of Lemma~\ref{changement de base} on the set $S_m + S_{m'}$. Indeed, thanks to the nested hypothesis on the linear sieves, we know that $S_m + S_{m'} = S_{m \vee m'}$.\\
Thanks to \eqref{eq:nutilde} we know that if $t=\sum_{\lambda \in \Lambda_{m \vee m'}} \tilde{a}_\lambda \tilde{\varphi}_\lambda$, with $(\tilde{\varphi}_\lambda)_{\Lambda_{m \vee m'}}$ an orthonormal basis of $S_{m\vee m'}$ for the norm $\|\cdot\|_{2,\sigma}$,
\[ \tilde{\nu}(l_t) = \sum_{\lambda \in \Lambda_{m \vee m'}} \tilde{a}_\lambda \tilde{\nu}(\tilde{\varphi}_\lambda) \leq \sqrt{\sum_{\lambda \in \Lambda_{m \vee m'}}\tilde{a}^2_\lambda} \sqrt{\sum_{\Lambda_{m \vee m'}}\tilde{\nu}^2(\tilde{\varphi}_\lambda)}.\]
Thus we have 
\[ \sup_{t\in B_{m, m'}^{2, \sigma}(0, 1)} \tilde{\nu}^2(l_t) \leq \sup_{\sum_{\lambda \in \Lambda_{m \vee m'}} a^2_\lambda = 1} \left(\sum_{\lambda \in \Lambda_{m \vee m'}} a^2_\lambda \right)\left(\sum_{\lambda \in \Lambda_{m \vee m'}}\tilde{\nu}^2(\tilde{\varphi}_\lambda)\right) = \sum_{\lambda \in \Lambda_{m \vee m'}}\tilde{\nu}^2(\tilde{\varphi}_\lambda).\]
We define for all $i\in\{ 1,\dots , n\}$ and $\lambda \in \Lambda_{m \vee m'}$, $\tilde{T}_{i, \lambda} = \Delta_i \tilde{\varphi}_\lambda(X_i)$ and $\tilde{Z}_{i, \lambda} = \tilde{T}_{i, \lambda} - \Esp(\tilde{T}_{i, \lambda}|L_i, U_i)$.
We notice $\sum_{\lambda \in \Lambda_{m \vee m'}}\tilde{\nu}^2(\tilde{\varphi}_\lambda) = \left(\frac{1}{n}\sum_{i=1}^n \tilde{Z}_{i, \lambda} \right)^2$ .\\
Using computations that lead to \eqref{eq:EspdeTetZ} we have $\Esp(\tilde{Z}^2_{i, \lambda}) \leq \Esp(\tilde{T}^2_{i, \lambda})$. Using this and Lemma~\ref{changement de base} we have
\begin{align*}
\Esp(\sup_{l\in\mathcal{L}} \tilde{\nu}^2(l)) &= \Esp\left(\sup_{t\in B_{m, m'}^{2,\sigma}(0, 1)} \tilde{\nu}^2(l_t)\right) \leq \sum_{\lambda \in \Lambda_{m \vee m'}} \Esp( \tilde{\nu}^2(\tilde{\varphi}_\lambda))\\
& \leq \sum_{\lambda \in \Lambda_{m \vee m'}}\Esp\left( \left[\frac{1}{n}\sum_{i=1}^n \tilde{Z}_{i, \lambda} \right]^2 \right) \\
& \leq \sum_{\lambda \in \Lambda_{m \vee m'}} \frac{1}{n}\Esp( \Delta^2 \tilde{\varphi}_\lambda^2(X) ) \leq \frac{1}{n} \sum_{\lambda \in \Lambda_{m \vee m'}} \Inte  \tilde{\varphi}_\lambda^2(x) f(x)\sigma(x) \dd x \\
&\leq \frac{1}{n} \Inte \left\|  \sum_{\lambda \in \Lambda_{m \vee m'}}  \tilde{\varphi}_\lambda^2 \right\|_{\infty} f(x)\sigma(x) \dd x \\
&\leq  \frac{1}{n} \Inte \|G_{m \vee m'}^{-1}\|_{op}\left\|  \sum_{\lambda \in \Lambda_{m \vee m'}}  \philambda^2 \right\|_{\infty} f(x)\sigma(x) \dd x \\
&\leq \frac{\Phi_0^2}{n} D_{m \vee m'}\|G_{m \vee m'}^{-1}\|_{op}\Inte f(x)\sigma(x) \dd x \\
&\leq \frac{\Phi_0^2}{n}(D_m\|G_m^{-1}\|_{op} + D_{m'}\|G_{m'}^{-1}\|_{op})\Inte f(x)\sigma(x) \dd x =: H^2 .
\end{align*}
  Having obtained $M_1, v$ and $H$, we are ready to get an upperbound for $\Esp (R_m)$ by applying Lemma~\ref{Talagrand_Conditionnel}. To this end, in the expression of $R_m$ given in \eqref{eq:Om}, we set for the function $\text{p}:(m, m') \mapsto 2(1+2\epsilon)\Phi_0^2 \frac{D_m\|G_m^{-1}\|_{op} + D_{m'}\|G_{m'}^{-1}\|_{op}}{n} \Inte f(x)\sigma(x) \dd x$. Choosing $\epsilon=\frac{1}{2}$ and since we need $x > 4$ thus to verify \eqref{eq:petpen} we choose the penalty term as \\$\text{pen}(m) = \kappa \Phi_0^2\|G_m^{-1}\|_{op} \frac{D_m}{n}\Inte f(x)\sigma(x)\dd x  $ for $\kappa:= 4x > 16$. Using Lemma~\ref{Talagrand_Conditionnel} we have that

\begin{align*}
\Esp(R_m) &\leq \sum_{m' \in \mathcal{M}'_n} \Esp\left( \sup_{t\in B} \tilde{\nu}(l_t)^2 - \text{p}(m, m')\right)\\
&\leq \sum_{m'\in\mathcal{M}'_n}\frac{\kappa_1}{n} \left( e^{-\kappa_2 (D_m\|G_m^{-1}\|_{op} + D_{m'}\|G_{m'}^{-1}\|_{op})} + \frac{e^{-\kappa_3 \epsilon^{3/2}\sqrt{n}\frac{\sqrt{D_m\|G_m^{-1}\|_{op} + D_{m'}\|G_{m'}^{-1}\|_{op}}}{\sqrt{D_m + D_{m'}}}}}{nC(\epsilon)^2} \right) \\
&\leq \frac{\kappa_1}{n}  \left(\sum_{m'\in\mathcal{M}'_n} e^{-\kappa_2 (D_m\|G_m^{-1}\|_{op} + D_{m'}\|G_{m'}^{-1}\|_{op})} + \frac{e^{-\kappa_3 \epsilon^{3/2}\sqrt{n}}}{nC(\epsilon)^2} \right) \\
&\leq \left(\frac{\kappa_1 e^{-\kappa_2 D_m\|G_m^{-1}\|_{op}}}{n} \underbrace{\sum_{k=0}^n  e^{-\kappa_2 k}}_{< + \infty}\right) + \frac{\kappa_1 e^{-\kappa_3 \epsilon^{3/2}\sqrt{n}}}{nC(\epsilon)^2} \\
&\leq \frac{C_1}{n},
\end{align*}
where we use that $\|G_m^{-1}\|_{op} \geq 1$ for all $m\in\mathcal{M}_n$, using \eqref{eq:lien norme 2sigma et norme 2}
\[\|G_m^{-1}\|_{op} = \sup_{t\in S_m} \frac{\|t\|_2^2}{\|t\|_{2, \sigma}^2} \geq \sup_{t\in S_m} \frac{\|t\|_2^2}{\|t\|_2^2} \geq 1 . \]
\\
Finally, recalling that on $\Xi_{\sigma_0}$, $\frac{\sigma_0}{2}\|\mathring{f}_{\hat{m}} - f\|_2^2 \leq \|\mathring{f}_{\hat{m}} - f\|_{n,\sigma}^2$ and inequality \eqref{Om},  we  have for all $m\in\mathcal{M}_n,$
\[ A_1 \leq \frac{2(x+2)}{\sigma_0(x-4)}\|f_m - f\|_2^2 + \frac{4x}{\sigma_0(x-4)}\text{pen}(m) + \frac{2x^2 C_1}{\sigma_0(x-4)}\frac{1}{n}. \]

For $A_2$ we use the fact that on $\Gamma_{m}^c, \tilde{f}_{m} \equiv 0 $. Using inequalities of Lemma~\ref{ineq_de_normes} and since we are in $\Xi_{\sigma_0}$ we know that if $\|t\|_2^2 > M$ then $\|t\|_{n, \sigma}^2 > \frac{\sigma_0 M}{2}$ for $t\in\mathbb{L}^2(\cercle)$ and $M>0$. Thus we have
\begin{align*}
A_2 &= \Esp(\|\tilde{f}_{\hat{m}} - f\|_2^2\mathds{1}_{\Xi_{\sigma_0}\cap \Gamma_{\hat{m}}^c})\\
&\leq \|f\|_2^2 \Proba(\Xi_{\sigma_0}\cap \Gamma_{\hat{m}}^c)\\
&\leq \|f\|_2^2 \Proba\left(\Xi_{\sigma_0} \cap \left\{\|\mathring{f}_{\hat{m}}\|_{n,\sigma}^2 > \frac{\sigma_0 k_n}{2}\right\}\right)\\
&\leq \|f\|_2^2 \frac{2}{\sigma_0 k_n}\Esp(\|\mathring{f}_{\hat{m}}\|_{n,\sigma}^2 \mathds{1}_{\Xi_{\sigma_0}}).
\end{align*}
Moreover thanks to the definition of $\mathring{f}_m$ and \eqref{eq:majoration normeop hatG} we have, for all $m\in\mathcal{M}_n$,
\[\|\mathring{f}_m\|_{n, \sigma}^2 = {}^t \hat{U}_m\hat{G}_m^{-1}\hat{G}_m\hat{G}_m^{-1} \hat{U} ={}^t \hat{U}_m\hat{G}_m^{-1} \hat{U}_m \leq \|\hat{G}_m^{-1}\|_{op}\|U_m\|_2^2 \leq \frac{1}{\pi\sigma_0}n .\]
Thus we have
\[A_2 \leq \|f\|_{\infty}^2 \frac{4\Phi_0^2 n}{\sigma_0^2 k_n} \Proba(\Xi_{\sigma_0}) \leq \|f\|_{\infty}^2 \frac{2n}{\pi\sigma_0^2 k_n} .\]
So if we take $k_n \geq n^2$ we have
\[A_2 \leq \|f\|_\infty^2 \frac{2}{\pi\sigma_0^2 n} .\]

For $A_3$ we know that $\|\tilde{f}_{\hat{m}}\|_2^2 \leq k_n$ so we have thanks to Lemma~\ref{Lemme} used with $y=\frac{\sigma_0}{2}$,
\[ A_3 \leq \Esp(2(\|f\|_2^2 + \|\tilde{f}_{\hat{m}}\|_2^2)\mathds{1}_{\Xi_{\sigma_0}^c}) \leq 2(\|f\|_2^2 +k_n)\Proba(\Xi_{\sigma_0}^c) \leq 12(\|f\|_2^2 + k_n)e^{-2n \frac{\sigma_0^2}{36}} \leq \frac{C_2}{n},\]
where $C_2$ depends on $\sigma_0$ and $\|f\|_{\infty}$. 
In the end we have for all $m\in\mathcal{M}_n$
\begin{align*}
\Esp(\|\tilde{f}_{\hat{m}} - f\|^2_2) &\leq \frac{2(x+2)}{\sigma_0(x-4)}\|f_m - f\|_2^2 + \frac{4x}{\sigma_0(x-4)}\text{pen}(m)  \\
&\quad +   \left(\frac{2x^2 C_1}{\sigma_0(x-4)} + \frac{2\|f\|_{\infty}^2}{\pi\sigma_0^2} + C_2\right)\frac{1}{n}\\
&\leq C(\sigma_0) \left(\|f_m - f\|_2^2 + \text{pen}(m) \right) + \frac{\tilde{C}(\sigma_0, \|f\|_{\infty})}{n},
\end{align*}
and thus we have
\[ \Esp(\|\tilde{f}_{\hat{m}} - f\|^2_2) \leq C(\sigma_0) \inf_{m\in\mathcal{M}_n}(\|f-f_m\|_2^2 + \text{pen}(m)) + \frac{\tilde{C}(\sigma_0, \|f\|_\infty)}{n},\]
which completes the proof of Proposition \ref{Adaptation du deuxieme estimateur}.
\end{proof}

\subsubsection{Proof of Theorem~\ref{Adaptation sans inconnues du deuxieme estimateur}}\label{Preuve_Adaptation_Sans_Inconnues_Deuxieme_Estimateur}

\begin{proof}

We use the same partition of $\Omega$ than in \eqref{eq:Separationenensemble}, so we have
\begin{equation}\label{eq:miseamajorer}
\Esp(\|\tilde{f}_{\hat{m}^*} - f\|^2_2)=\underbrace{\Esp(\|\tilde{f}_{\hat{m}^*} - f\|^2_2\mathds{1}_{\Xi_{\sigma_0}\cap \Gamma_{\hat{m}^*}})}_{=A_1} + \underbrace{\Esp(\|\tilde{f}_{\hat{m}^*} - f\|^2_2\mathds{1}_{\Xi_{\sigma_0}\cap \Gamma_{\hat{m}^*}^c})}_{=A_2} + \underbrace{\Esp(\|\tilde{f}_{\hat{m}^*} - f\|^2_2\mathds{1}_{\Xi_{\sigma_0}^c})}_{=A_3} .
\end{equation}
We consider the following set
\[T = \{\omega\in\Omega, \forall m  \in\mathcal{M}_n, \|G_m^{-1}\|_{op} \leq 2\|\hat{G}_m^{-1}\|_{op}\},\]
and introduce it in the term $A_1$ of \eqref{eq:miseamajorer}. This splits $A_1$ into two terms giving
\begin{align*}
\Esp(\|\tilde{f}_{\hat{m}^*} - f\|^2_2) &=\underbrace{\Esp(\|\tilde{f}_{\hat{m}^*} - f\|_2^2\mathds{1}_{\Xi_{\sigma_0}\cap \Gamma_{\hat{m}^*}\cap T})}_{=A_{11}} +  \underbrace{\Esp(\|\tilde{f}_{\hat{m}^*} - f\|_2^2\mathds{1}_{\Xi_{\sigma_0}\cap \Gamma_{\hat{m}^*}\cap T^c })}_{=A_{12}} \\
 &+ \underbrace{\Esp(\|\tilde{f}_{\hat{m}^*} - f\|^2_2\mathds{1}_{\Xi_{\sigma_0}\cap \Gamma_{\hat{m}^*}^c})}_{=A_2} + \underbrace{\Esp(\|\tilde{f}_{\hat{m}^*} - f\|^2_2\mathds{1}_{\Xi_{\sigma_0}^c})}_{=A_3} .
\end{align*}
Since $A_2$ and $A_3$ are the same quantities as in Section~\ref{Preuve_Adaptation_Deuxieme_Estimateur}, using the same arguments we have
\[ A_2 \leq \|f\|_{\infty}^2 \frac{2}{\pi\sigma_0^2 n} \hspace{15pt} \text{ and } \hspace{15pt} A_3 \leq 12(\|f\|_{\infty}^2 + n^2)e^{-2n \frac{\sigma_0^2}{36}}\leq \frac{C_2}{n}.\]
So it remains to control the terms $A_{11}$ and $A_{12}$.\\
For $A_{11}$ we use that on the set $T$, $\|G_m^{-1}\|_{op} \leq 2\|\hat{G}_m^{-1}\|_{op}$ and on $\Xi_{\sigma_0}, \|\hat{G}_m^{-1}\|_{op} \leq 2\|G_m^{-1}\|_{op}$ for all $m\in \mathcal{M}_n$ as stated in \eqref{eq:majoration normeop hatG}.
Thus we have $\Esp(\|\hat{G}_m^{-1}\|_{op}\mathds{1}_{\Xi_{\sigma_0}}) \leq 2\|G_m^{-1}\|_{op}$.
Using the same proof as in Section~\ref{Preuve_Adaptation_Deuxieme_Estimateur}, taking $\epsilon=1/2$ and  $x>4$, writing $\kappa := 8x > 32$, we set the penalization term to $\widehat{\text{pen}}(m) = \kappa\frac{\|\hat{G}_m^{-1}\|_{op}}{2\pi}\frac{D_m}{n}$. We obtain for any $m\in\mathcal{M}_n$,
\[ A_{11} \leq \frac{2(x+2)}{\sigma_0(x-4)}\|f_m - f\|_2^2 + \frac{8 x}{\pi\sigma_0^2(x-4x)}\frac{\kappa D_m}{n}  + \frac{2x^2}{\sigma_0(x-4)}\frac{C_1}{n}.  \]

For $A_{12}$, we have
\[ A_{12} \leq 2(n^2 + \|f\|_{\infty}^2) \Proba(T^c),\]
with the set $T^c$ defined as follows
\[ T^c = \{ \omega\in\Omega, \exists m \in \mathcal{M}_n, \|G_m^{-1}\|_{op} > 2\|\hat{G}_m^{-1}\|_{op} \}.\]
Thus we have that 
\[ \Proba(T^c) \leq \sum_{m\in\mathcal{M}_n} \Proba( \|G_m^{-1}\|_{op} > 2\|\hat{G}_m^{-1}\|_{op}).\]
But for any $m$ we know that $\|G_m^{-1}\|_{op} \leq \|G_m^{-1} - \hat{G}_m^{-1}\|_{op} + \|\hat{G}_m^{-1}\|_{op}$. 
So if $\|G_m^{-1}\|_{op}> 2\|\hat{G}_m^{-1}\|_{op}$ then we have $\|G_m^{-1} - \hat{G}_m^{-1}\|_{op} > \|\hat{G}_m^{-1}\|_{op} \geq \sigma_0\|G_m^{-1}\|_{op}$ where last inequality is obtained using \eqref{eq: minoration norme op hatGm}.
Thus we have
\[ A_{12} \leq 2(n^2+\|f\|_2^2)\sum_{m\in\mathcal{M}_n}\Proba(\|G_m^{-1} - \hat{G}_m^{-1}\|_{op} > \sigma_0\|G_m^{-1}\|_{op}).\]
Finally using Lemma~\ref{lemme pour la fin} with $\alpha = \sigma_0 \leq 1$ we have
\[ \Proba(\|G_m^{-1} - \hat{G}_m^{-1}\|_{op} > \sigma_0\|G_m^{-1}\|_{op}) \leq \Proba\left(\Xi_{\sigma_0^2}^c\right) \leq 6e^{-2n\frac{(\sigma_0)^4}{36}} \leq \frac{C_3}{n^4},\]
using Lemma~\ref{Lemme} for the penultimate upper bound.
So we have
\[ A_{12} \leq 2(n^2+\|f\|_2^2)|\mathcal{M}_n|\frac{C_3}{n^4} \leq \frac{C_4}{n}.\]
In consequence,  we eventually obtain for all $m\in \mathcal{M}_n$
\begin{align*}
\Esp(\|\tilde{f}_{\hat{m}^*} - f\|^2_2)  &\leq  \frac{2(x+2)}{\sigma_0(x-4)}\|f_m - f\|_2^2 + \frac{8 x}{\pi\sigma_0^2(x-4x)}\frac{\kappa D_m}{n}  \\
&\quad +\left( \frac{2x^2 C}{\sigma_0(x-4)}  + \frac{2 \|f\|_2^2}{\pi\sigma_0^2} + C_2 + C_4\right)\frac{1}{n}\\
&\leq K(\sigma_0) \left( \|f_m - f\|_2^2 + \frac{\kappa D_m}{n}  \right) + \frac{\tilde{K}(\sigma_0, \|f\|_{\infty})}{n}.
\end{align*}
This means
\[ \Esp(\|\tilde{f}_{\hat{m}^*} - f\|^2_2) \leq K(\sigma_0) \inf_{m\in\mathcal{M}_n}\left(\|f-f_m\|_2^2 + \kappa\frac{D_m}{n}\right) + \frac{\tilde{K}(\sigma_0, \|f\|_\infty)}{n},\]
which is the desired result. 

\end{proof}

\section*{Bibliography}

\printbibliography[heading=none]

\end{document}